\documentclass{amsart}
\usepackage{color}
\usepackage{enumerate}
\usepackage{amsthm,amssymb}
\usepackage{amsmath}
\usepackage{graphicx,epsfig}
\usepackage{hyperref}
\def\CC{{\mathbb{C}}}

\def\R{{\mathbb{R}}}
\def\NN{{\mathbb{N}}}

\newtheorem{theo}{Theorem}[section]

\newtheorem{prop}[theo]{Proposition}
\newtheorem{cor}[theo]{Corollary}

\newtheorem{hyp}[theo]{Hypothesis}
\theoremstyle{definition}

\newtheorem{rem}[theo]{Remark}
\newtheorem{defi}[theo]{Definition}

\def\calL{{\mathcal{L}}}

\def\calA{{\mathcal{A}}}

\def\R{\mathbb R}

\def\N{\mathbb N}

\newcommand{\norm}[1]{\left\Vert#1\right\Vert}
\newcommand{\abs}[1]{\left\vert#1\right\vert}

\let\div\undefined
\DeclareMathOperator{\div}{\mathrm{div}}

\newcommand{\bpr}{\begin{proof}}  
\newcommand{\epr}{\end{proof}}
\newcommand{\eps}{\varepsilon}

\begin{document}

\numberwithin{equation}{section}

\title[Kernel estimates of Schr\"odinger type operators]{General Kernel estimates of Schr\"odinger type operators with unbounded diffusion terms}
\author{Loredana Caso, Markus Kunze, Marianna Porfido and Abdelaziz Rhandi}
\address{M. Kunze: Fachbereich Mahematik und Statistik, Universit\"at Konstanz, 78457 Konstanz, Germany}
\address{L. Caso, M. Porfido, A. Rhandi: Dipartimento di Matematica, Università degli Studi di Salerno, Via Giovanni Paolo II, 132, 84084 Fisciano (Sa), Italy}
\thanks{The first, the third and the fourth authors are members of the
Gruppo Nazionale per l'Analisi Matematica, la Probabilità e le loro Applicazioni (GNAMPA) of
the Istituto Nazionale di Alta Matematica (INdAM). This article is based upon work
from COST Action 18232 MAT-DYN-NET, supported by COST (European Cooperation in Science
and Technology), www.cost.eu}

\keywords{Schr\"odinger type operator, unbounded coefficients, kernel estimates, ultracontractive semigroup}
\subjclass[2020]{35K08, 35P15, 35K10}

\begin{abstract}
We prove first that the realization $A_{\min}$ of $A:={\div}(Q\nabla)-V$ in $L^2(\R^d)$ with unbounded coefficients generates a symmetric sub-Markovian and ultracontractive semigroup on $L^2(\R^d)$ which coincides on $L^2(\R^d)\cap C_b(\R^d)$ with the minimal semigroup generated by a realization of $A$ on $C_b(\R^d)$. Moreover, using time dependent Lyapunov functions, we prove pointwise upper bounds for the heat kernel of $A$ and deduce some spectral properties of $A_{\min}$ in the case of polynomially and exponentially diffusion and potential coefficients.
\end{abstract}

\maketitle

\section{Introduction}

In this article, we are concerned with Schr\"odinger type operators of the form
\begin{equation}\label{eq.div}
A\varphi ={\div}(Q\nabla \varphi)-V\varphi , \quad \varphi \in C^2(\R^d),
\end{equation}
where the diffusion coefficients $Q$, and the potential $V$ are typically unbounded functions. Throughout, we make the following assumptions on $Q$ and $V$.

\begin{hyp}\label{h.1}
We have $Q=(q_{ij})_{i,j=1, \ldots, d} \in C^{1+\zeta}(\R^d; \R^{d\times d})$ and $0\leq V \in C^\zeta(\R^d)$ for some $\zeta \in (0,1)$. Moreover, 
\begin{enumerate}
\item The matrix $Q$ is symmetric and uniformly elliptic, i.e. there is $\eta>0$ such that
\[
\sum_{i,j=1}^dq_{ij}(x)\xi_i\xi_j\ge \eta |\xi|^2 \quad \hbox{for all }x,\,\xi\in \R^d;
\]
\item there is $0\le Z\in C^2(\R^d)$ and a constant $M\ge 0$ such that $\lim_{|x|\to \infty}Z(x)=\infty$, $AZ(x)\le M$ and $\eta\Delta Z(x)-V(x)Z(x)\leq M$ for all $x\in \R^d$.
\end{enumerate}
\end{hyp}

In these last years  second order elliptic operators with polynomially growing coefficients and their associated semigroups have recived a lot of attention, see for example \cite{boutiah_et_al}, \cite{BRC19}, \cite{BCGT21}, \cite{AC-AR-CT}, \cite{AC2-AR2-CT2}, \cite{F-L}, \cite{G-S}, \cite{G-S-3}, \cite{Luca-Abde}, \cite{Me-Sp-Ta}, \cite{Me-Sp-Ta2}, \cite{Lorenzi} and the references therein.

Concerning the above operator $A$, it is well known (see \cite[Theorem 2.2.5]{Lorenzi} and \cite{MPW}) that, assuming Hypothesis \ref{h.1}, 
a suitable realization of $A$ generates a semigroup $T=(T(t))_{t\geq 0}$ on the space $C_b(\R^d)$ that is given through an integral kernel; more precisely, 
\[
T(t)f(x)=\int_{\R^d}p(t,x,y)f(y)\,dy,\quad t>0,\,x\in \R^d,\,f\in C_b(\R^d),
\]
where the kernel $p$ is positive, $p(t,\cdot ,\cdot)$ and $p(t,x,\cdot)$ are measurable for any $t>0,\,x\in \R^d$, and for a.e. fixed $y\in \R^d,\,p(\cdot ,\cdot ,y)\in C_{\mathrm{loc}}^{1+\frac{\zeta}{2},2+\zeta}((0,\infty)\times \R^d)$.\smallskip

As it can be seen in Section \ref{S2}, this semigroup can be extended to a symmetric sub-Markovian and ultracontractive $C_0$-semigroup on $L^2(\R^d)$ and classical results show that this semigroup extrapolates to a positive $C_0$-semigroup of contractions in all $L^p(\R^d)$, $p\in [1,\infty)$. Moreover in the examples considered in Section \ref{S4}, these semigroups are compact and the spectrum of their corresponding generators are independent of $p$.

Our second focus in this article lies in proving pointwise upper bounds for the kernel $p$. The case of (non-divergence type) Schr\"odinger operators
\begin{equation}
\label{eq.nondiv}
(1+|x|^m)\Delta - |x|^s
\end{equation}
was discussed extensively in the literature and may serve as a model case. In this case, kernel estimates were obtained in \cite{AC2-AR2-CT2}
(see also \cite{BRC19} from which kernel estimates for the corresponding divergence form operators can be deduced) assuming that $m>2$ and $s>m-2$. The case $m\in [0,2)$ and $s>2$ was treated in \cite{Luca-Abde}. Let us also mention that for $m=0$ and $s>0$ both upper and lower estimates were established in \cite{Meta-Sp}. In the case of  $V\equiv 0$, similar kernel estimates were obtained in \cite{Me-Sp-Ta}.

As far as more general operators are concerned, in particular the case of bounded  diffusion coefficients has recived a lot of attention, see
\cite{ALR10}, \cite{BRS06}, \cite{LMPR11}, \cite{MPR}. These techniques were extended to include also unbounded diffusion coefficients in 
\cite{KunzeLorenziRhandi1, KunzeLorenziRhandi2}.\smallskip

In this article, we adopt the technique of time dependent Lyapunov functions used in \cite{Spina08}, \cite{ALR10}, \cite{KunzeLorenziRhandi1}, \cite{KunzeLorenziRhandi2} to our divergence form setting. This allows for a unified approach to obtain kernel bounds corresponding to \cite{BRC19}, \cite{Meta-Sp} in the divergence form setting. As a matter of fact, we can allow even more general conditions on $m$ and $s$, requiring merely that $m>0$ and $s>|m-2|$; moreover, we can drop the assumption $d\geq 3$ imposed in \cite{BRC19, Meta-Sp}.

As our approach does not depend on the specific structure of the coefficients, we can establish kernel estimates not only in the case where $Q(x) = (1+|x|^m)I$; an estimate of the quadratic form associated to $Q$ is enough, cf.\ Equation \eqref{Hp1 Prop: Time dependent Lyapunov functions in case of polynomially growing diffusions}. Moreover, we can even leave the setting of polynomially growing  coefficients and consider coefficients of \emph{exponential growth}; this includes the case $Q(x) = e^{|x|^m}I$ and $V(x) = e^{|x|^s}$ for $d\geq 1$ and $2\leq m < s$. Here we would like to mention the paper \cite{FFMP} where pointwise estimates are obtained in the elliptic case for exponentially growing coefficients. We stress that these estimates can be improved by choosing a Lyapunov function as in Subsection \ref{S52}. \smallskip

This article is organized as follows. In Section \ref{S2},  we adapt the techniques in \cite{AMP} to prove that a realization of $A$ in $L^2(\R^d)$ generates a symmetric sub-Markovian and ultracontractive semigroup $T_2(\cdot)$ on $L^2(\R^d)$ which coincides with the semigroup $T(\cdot)$ on $L^2(\R^d)\cap C_b(\R^d)$. In Section \ref{S3} we introduce time dependent Lyapunov functions and establish sufficient conditions under which certain exponential functions are time dependent Lyapunov functions in the case of polynomially and exponentially growing diffusion coefficients. In the subsequent Section \ref{S4}, we use these results to prove upper kernel estimates for our divergence form operator $A$. In the concluding Section \ref{S53}, we present some consequences of our result for the spectrum and the eigenfunctions of the operator $A_{\min}$ from Section \ref{S2}.

\subsection*{Notation}
$B_r$ denotes the open ball of $\R^d$ of radius $r$ and center $0$. For $0\leq a<b$, we write $Q(a,b)$ for $(a,b)\times \R^d$.

If $u: J\times \R^d\to \R$, where $J\subset [0,\infty)$ is an
interval, we use the following notation:
\begin{align*}
\partial_t u =&\frac{\partial u}{\partial t}, \
D_iu=\frac{\partial u}{\partial x_i}, \
D_{ij}u=D_iD_ju
\\
\nabla u=&(D_1u, \dots, D_du), \ {\div}(F)=\sum_{i=1}^dD_iF_i \hbox{\ for }F:\R^d\to \R^d,
\end{align*}
and
$$
|\nabla u|^2=\sum_{j=1}^d |D_j u|^2 , \qquad
|D^2u|^2=\sum_{i,j=1}^N|D_{ij} u|^2.
$$
Let us come to notation for function spaces.  $C_b(\R^d)$ is the
space of bounded and continuous functions in $\R^d$.  $\mathcal{D}(\R^d)$ is the space
of test functions. $C^\alpha(\R^d)$ denotes the space of all
$\alpha$-H\"older continuous functions on $\R^d$. $C^{1,2}(Q(a,b))$ is the space of all functions $u$ such that $\partial_tu$, $D_iu$ and $D_{ij}u$ are continuous in $Q(a,b)$.

 For $\Omega \subseteq \R^d,\,1\le k\le \infty,\,j\in \N,\,W^{j}_{k}(\Omega)$
denotes the classical Sobolev space of all $L^k$--functions having
weak derivatives in $L^k(\Omega)$ up to the order $j$.  Its usual
norm is denoted by $\|\cdot \|_{j,k}$ and by $\|\cdot \|_k$ when
$j=0$. When $k=2$ we set $H^j(\Omega):=W^{j}_{2}(\Omega)$ and $H_0^1(\Omega)$ denotes the closure of 
the set of test functions on $\Omega$ with respect to the norm of $H^1(\Omega)$.

For $0<\alpha \le 1$ we denote by
$C^{1+\alpha/2 ,2+\alpha}(Q(a,b))$ the space of all
functions $u$ such that $\partial_tu$, $D_iu$ and $D_{ij}u$ are 
$\alpha$-H\"older continuous in $Q(a,b)$ with respect to
the parabolic distance $d((t,x),(s,y)):=|x-y|+|t-s|^{\frac{1}{2}}$.
Local H\"older spaces are defined, as usual, requiring that the
H\"older condition holds in every compact subset.

\section{Generation of semigroups on \texorpdfstring{$L^2(\R^d)$}{}}\label{S2}

In this section we show that a realization of $A$ in $L^2(\R^d)$ generates a symmetric sub-Markovian and ultracontractive semigroup $T_2(\cdot)$ on $L^2(\R^d)$ which coincides with the semigroup $T(\cdot)$ on $L^2(\R^d)\cap C_b(\R^d)$.

We recall that, given $\Omega \subset \R^d$, a $C_0$-semigroup $S(\cdot)$ on $L^2(\Omega)$ is called \emph{sub-Markovian} if $S(\cdot)$ is a positive semigroup, i.e.\ $S(t)f\geq 0$ for all $t\geq 0$ and $f\geq 0$, and \emph{$L^\infty$-contractive}, i.e.
\[
\|S(t)f\|_\infty \le \|f\|_\infty,\quad \forall t\ge 0,\,f\in L^2(\Omega)\cap L^\infty(\Omega).
\]
It is called \emph{ultracontractive}, if there is a constant $c>0$ such that
\[
\|S(t)\|_{\calL(L^1, L^\infty)}\leq c t^{-\frac{d}{2}},
\]
for all $t>0$.

To establish ultracontractivity we use the following useful result, see \cite[Proposition 1.5]{AMP}, where we replace the $H^1$-norm with the $L^2$-norm of the gradient. The proof remains the same and is based on Nash's inequality
\[
\|u\|^{1+\frac{2}{d}}_2 \leq c_d\| |\nabla u|\|_2\|u\|_1^{\frac{2}{d}}
\]
for all $u\in L^1(\R^d)\cap H^1(\R^d)$.

\begin{prop}\label{ultra-arendt}
Let $S(\cdot)$ be a $C_0$-semigroup on $L^2(\R^d)$ such that $S(\cdot)$ and $S^*(\cdot)$ are sub-Markovian. Assume that, for $\delta>0$, the generator $B$ of $S(\cdot)$ satisfies
\begin{enumerate}
\item $D(B)\subset H^1(\R^d)$;
\item $\langle -Bu,u\rangle\geq \delta \||\nabla u|\|_2^2 ,\quad \forall u\in D(B)$;
\item $\langle -B^*u,u\rangle\geq \delta \||\nabla u|\|_2^2 ,\quad \forall u\in D(B^*)$.
\end{enumerate}
Then there is $c_\delta >0$ such that
\[
\norm{S(t)}_{\calL(L^1,L^\infty)}\leq c_\delta t^{-d/2} ,\quad \forall t>0,
\]
i.e.\ $S(\cdot)$ is ultracontractive.
\end{prop}

We now take up our main line of study and consider the elliptic operator $\calA$, defined by
\[
\calA : H^1_{\rm{loc}}(\R^d)\to\mathcal{D}(\R^d)',\quad \mathcal{A}\varphi={\div}(Q\nabla \varphi)-V\varphi.
\]
Its maximal realization $A_{\max}$ in $L^2(\R^d)$ is defined by
\begin{align*}
D(A_{\max}) & = \lbrace u\in L^2(\R^d)\cap H^1_{\rm{loc}}(\R^d), \,\calA u\in L^2(\R^d)\rbrace,\\
A_{\max} u & = \calA u.
\end{align*}
There is also a minimal realization $A_{\min}$ of $\calA$.
The minimal realization of $\calA$ in $L^2(\R^d)$ is the operator presented in the following theorem.

\begin{theo}\label{Thm: minimal realization}
There exists a unique operator $A_{\min}$ on $L^2(\R^d)$ such that 
\begin{enumerate}
\item $A_{\min}\subset A_{\rm max}$;
\item $A_{\min}$ generates a positive, symmetric $C_0$-semigroup $T_2(\cdot)$ on $L^2(\R^d)$;
\item if $B\subset A_{\rm max}$ generates a positive $C_0$-semigroup $S(\cdot)$, then $T_2(t)\leq S(t)$ for all $t\geq 0$.
\end{enumerate}
The operator $A_{\min}$ and the semigroup $T_2(\cdot)$ have the following additional properties:
\begin{enumerate}
\setcounter{enumi}{3}
\item $D(A_{\min})\subset H^1(\R^d)$ and $-\langle A_{\min}u,u\rangle \geq \eta \| |\nabla u|\|_2^2$ for all $u\in D(A_{\min})$;
\item $T_2(\cdot)$ is sub-Markovian and ultracontractive;
\item  the semigroup $T_2(\cdot)$ is consistent with $T(\cdot)$, i.e.
\[
T_2(t)f=T(t)f,\quad t\ge 0,\,f\in L^2(\R^d)\cap C_b(\R^d).
\]
\end{enumerate}
\end{theo}

\begin{proof}
We adapt the proof of Theorem 1.1, Proposition 1.2 and Proposition 1.3 in \cite{AMP} to our situation. For the reader’s convenience we provide  the details.

\emph{Step 1.} We define approximate semigroups $T^{(\rho)}(\cdot)$ on $L^2(B_\rho)$.

To that end, consider the bilinear form $\mathfrak{a}_\rho : H^1_0(B_\rho)\times H^1_0(B_\rho) \to \CC$, defined by
\[
\mathfrak{a}_\rho [u,v]= \int_{B_\rho} \sum_{i,j=1}^d q_{ij}D_i u D_j\bar v\, dx + \int_{B_\rho} Vu\bar v\, dx.
\]
This form is obviously symmetric. Using that $Q$ and $V$ are bounded on $B_\rho$, an easy application of H\"older's inequality shows that
$\mathfrak{a}_\rho$ is continuous. Moreover, the positivity of $V$, the uniform ellipticity of $Q$ and Poincar\'e's inequality yield coercivity
of $\mathfrak{a}_\rho$. Now standard theory, see \cite[Proposition 1.51]{Ouhabaz} implies that the associated operator $A_\rho$ generates
a strongly continuous semigroup $T^{(\rho)}(\cdot)$ on $L^2(B_\rho)$. Making use of the Beurling--Deny criteria (see, e.g., Corollary 4.3 and Theorem 4.7 in \cite{Ouhabaz}) we see that the semigroup $T^{(\rho)}(\cdot)$ is sub-Markovian.\smallskip

\emph{Step 2.} We prove that the semigroups $T^{(\rho)}(\cdot)$ are increasing to a semigroup $T_2(\cdot)$.

We now consider functions on $B_\rho$ to be defined on all of $\R^d$, by extending them with $0$ outside of $B_\rho$. Then, for any
$0<\rho_1 <\rho_2$, the space $H^1_0(B_{\rho_1})$ is an ideal in $H^1_0(B_{\rho_2})$. Thus, by \cite[Corollary B.3]{St-Vo} (see also \cite[Section 2.3]{Ouhabaz}), we have $T^{(\rho_1)} (t) \leq T^{(\rho_2)}(t)$ for all $t\geq 0$. As every semigroup $T^{(\rho)}(\cdot)$ is sub-Markovian and thus contractive, we may define
\[
T_2(t)f  := \sup_{n\in \NN} T^{(n)}(t)f
\]
for $0\leq f\in L^2(\R^d)$ and then $T_2(t)f := T_2(t)f^+-T_2(t)f^-$ for general $f\in L^2(\R^d)$. An easy computation shows that
$T_2(\cdot)$ is a positive contraction semigroup. We prove that $T_2(\cdot)$ is strongly continuous. To that end, fix
 $0\leq f\in \mathcal{D}(\R^d)$, and $\rho>0$ such that $\mathrm{supp} f\subset B_\rho$. Let $t_n\downarrow 0$. Then 
\begin{align*}
&\limsup_{n\to\infty} \|T^{(\rho)}(t_n)f-T_2(t_n)f\|_2^2\\
= &\limsup_{n\to\infty}\left[\|T^{(\rho)}(t_n)f\|_2^2 + \|T_2(t_n)f\|_2^2 - 2 \langle T^{(\rho)}(t_n)f, T_2(t_n)f\rangle_2\right]\\
\leq & \limsup_{n\to\infty} \left[2\norm{f}_2^2 -2 \langle T^{(\rho)}(t_n)f, T^{(\rho)}(t_n)f\rangle_2\right]= 2\|f\|_2^2 - 2\|f\|_2^2=0.
\end{align*}
Here, in the third line we have used the contractivity of $T^{(\rho)}(\cdot)$ and $T_2(\cdot)$, that  $0\leq T^{(\rho)}(t_n)f \leq T_2(t_n)f$ and the strong continuity of $T^{(\rho)}(\cdot)$. Thus, $T_2(t_n)f\to f$ as $n\to \infty$. Splitting $f\in \mathcal{D}(\R^d)$ into positive and negative part, 
we see that this is true for general $f$. In view of the contractivity of $T_2(\cdot)$, a standard $3\eps$ argument yields strong continuity of $T_2(\cdot)$.

As the form $\mathfrak{a}_\rho$ is symmetric, the semigroup $T^{(\rho)}(\cdot)$ consists of symmetric operators and thus, so does the limit semigroup $T_2(\cdot)$. Likewise, sub-Markovianity of $T_2(\cdot)$ is inherited by that of $T^{(\rho)}(\cdot)$.
\smallskip

\emph{Step 3.} We identify the generator $A_{\min}$ of $T_2(\cdot)$.

Let us first note that $R(\lambda, A_\rho)f \to R(\lambda, A_{\min})f$ as $\rho\to \infty$ for every $\lambda >0$; this follows from the construction of $T_2(\cdot)$ by taking Laplace transforms and using dominated convergence. Now fix a sequence $\rho_n\uparrow \infty$ and $f\in L^2(\R^d)$. We put $u=R(1, A_{\min})f$ and $u_n = R(1, A_{\rho_n})f$. Then $u_n\to u$ and $A_{\rho_n}u_n = u_n - f \to u-f = A_{\min}u$
in $L^2(\R^d)$ as $n\to \infty$. By coercivity of the form $\mathfrak{a}_{\rho_n}$, we have
\begin{equation}\label{eq.limsup}
\eta \limsup_{n\to\infty} \int |\nabla u_n|^2\, dx \leq \limsup_{n\to\infty} \mathfrak{a}_{\rho_n}[u_n, u_n] =
\limsup_{n\to\infty} - \langle A_{\rho_n} u_n, u_n\rangle = - \langle A_{\min}u, u\rangle.
\end{equation}

It follows that $(u_n)_{n\in\N}$ is a bounded sequence in $H^1(\R^d)$ and thus, by reflexivity of $H^1(\R^d)$,
 $u_n\to u$ weakly in $H^1(\R^d)$. Thus, $D(A_{\min})\subset H^1(\R^d)$. Moreover, using the weak lower semicontinuity of norms, we see that \eqref{eq.limsup} implies $-\langle A_{\min} u, u\rangle \geq \eta \||\nabla u|\|_2^2$. 
 
Now fix $v\in \mathcal{D}(\R^d)$. As $u_n$ converges to $u$ weakly in $H^1(\R^d)$, we see that
\[
\langle \calA u, v\rangle = \lim_{n\to\infty} \langle \calA u_n, v\rangle = \lim_{n\to\infty} \langle A_{\rho_n}u_n, v\rangle = \langle A_{\min}u , v\rangle,
\]
proving $A_{\min}\subset A_{\max}$. At this point, properties (a), (d) and (by definition of $A_{\min}$) (b) are proved.\smallskip

\emph{Step 4}. We establish the minimality property.

Thus,  let $B\subset A_{\max}$ such that $B$ generates a positive $C_0$-semigroup $S(\cdot)$ on $L^2(\R^d)$. To prove
$T_2(t)\leq S(t)$ for all $t\geq 0$ it suffices to prove $R(\lambda, A_{\min}) \leq R(\lambda, B)$ for all $\lambda >0$; this is an easy consequence of Euler's formula. 

To see this, let us fix again a sequence $\rho_n \uparrow \infty$, $\lambda >0$ and $0\leq f\in L^2(\R^d)$. We put
$u= R(\lambda, A_{\min})f$, $v= R(\lambda, B)f$ and $u_n = R(\lambda, A_{\rho_n})f$. As $B \subset A_{\max}$, we have 
$v\in H^1_{\mathrm{loc}}(\R^d)$ and 
\begin{equation}\label{eq.variational}
\lambda\int_{B_{\rho_n}}(u_n-v) w\, dx+ \int_{B_{\rho_n}} \sum_{i,j=1}^d q_{ij}D_j(u_n-v)D_iw\, dx + \int_{B_{\rho_n}} V(u_n-v) w\, dx = 0
\end{equation}
for all $w\in H^1_0(B_{\rho_n})$. As the semigroup $S(\cdot)$ is positive, $v \geq 0$ and thus $(u_n -v)^+ \leq u_n$. As $H^1_0(B_{\rho_n})$ is an Ideal in $H^1_{\mathrm{loc}}(\R^d)$, $(u_n-v)^+\in H^1_0(B_{\rho_n})$. We may thus insert $w=(u_n-v)^+$ into \eqref{eq.variational}. Taking the uniform ellipticity of $Q$ into account, this yields
\[
\lambda \int_{B_{\rho_n}} \big( (u_n -v)^+\big)^2\, dx + \eta \int_{B_{\rho_n}} |\nabla (u_n - v)^+|^2\, dx + \int_{B_{\rho_n}} V\big((u_n-v)^+\big)^2\, dx \leq 0.
\]
As $V\geq 0$, it follows that $(u_n - v)^+=0$ and thus $u_n \leq v$. Upon $n\to \infty$ we obtain $u\leq v$ and thus
$R(\lambda, A_{\min})f \leq R(\lambda, B)f$ for $0\leq f\in L^2(\R^d)$.\smallskip

\emph{Step 5}. We establish properties (e) and (f).

As we have already mentioned above, the semigroup $T_2(\cdot)$ is sub-Markovian and consists of symmetric operators. The latter implies that
the generator $A_{\min}$ of $T_2(\cdot)$ is selfadjoint. In view of property (d), the ultracontractivity of the semigroup follows immediately from
Proposition \ref{ultra-arendt}.

As for consistency we note that the semigroup $T(\cdot)$ on $C_b(\R^d)$ is obtained by the same approximation procedure as for $T_2(\cdot)$, see
\cite[Theorem 2.2.1]{Lorenzi}. But the semigroup solutions of the Cauchy--Dirichlet problem associated with $\calA$ on $C_b(\overline{B_{\rho}})$ is consistent with the semigroup solution on $L^2(B_{\rho})$ considered above. Thus, consistency of $T_2(\cdot)$ and $T(\cdot)$ follows.
\end{proof}

\begin{rem}
\begin{enumerate}
\item 
As the minimal realization $A_{\min}$ of the elliptic operator $\calA$ generates a symmetric sub-Markovian $C_0$-semigroup $T_2(\cdot)$ on $L^2(\R^d)$, it follows from  \cite[Theorem 1.4.1]{Davies}, that $T_2(\cdot)$ extends to a positive $C_0$-semigroup of contractions $T_p(\cdot)$ on $L^p(\R^d)$ for all $p\in [1,\infty)$. Moreover these semigroups are consistent, i.e.
\[
T_p(t)f=T_q(t)f,\quad \hbox{\ for all }f\in L^p(\R^d)\cap L^q(\R^d),\,t\ge 0.
\]
\item
Since, by Theorem \ref{Thm: minimal realization}, $T_2(\cdot)$ is ultracontractive, and $T_2(\cdot)$ coincindes with $T(\cdot)$ on $L^2(\R^d)\cap C_b(\R^d)$, it follows that $T_2(\cdot)$ is given through an integral kernel which coincides with the kernel $p$ of the semigroup $T(\cdot)$.
\end{enumerate}
\end{rem}

\section{Time dependent Lyapunov functions for parabolic operators with polynomially and exponentially diffusion coefficients}\label{S3}

As in \cite{Spina08}, \cite{ALR10} and \cite{KunzeLorenziRhandi2} we use time dependent Lyapunov functions to prove pointwise bounds of the kernel $p$. In this section we give conditions under which certain exponentials are time dependent Lyapunov functions
for $L:=\partial_t+A$ also in the case of polynomially and exponentially growing diffusion coefficients.

We now introduce, as in \cite{Spina08} and \cite{KunzeLorenziRhandi2}, time dependent Lyapunov functions for $L$.
\begin{defi}
We say that a function $W:[0,1]\times \R^d\to [0,\infty)$ is a \emph{time dependent Lyapunov function for} $L$ if $W\in C^{1,2}((0,1)\times \R^d)\cap C([0,1]\times \R^d)$ such that $\lim_{|x|\to \infty}W(t,x)=\infty$ uniformly for $t$ in compact subsets of $(0,1],\,W\le Z$ and there is $0\le h\in L^1(0,1)$ such that
\begin{equation}\label{eq1: definition time dependent Lyapunov functions}
LW(t,x)\le h(t)W(t,x)
\end{equation}
and 
\begin{equation}\label{eq2: definition time dependent Lyapunov functions}
\partial_t W(t,x)+\eta\Delta W(t,x)-V(x)W(t,x)\leq h(t)W(t,x)
\end{equation}
for all $(t,x)\in (0,1)\times \R^d$. To emphasize the dependence on $Z$ and $h$, we also say that $W$ is a time dependent Lyapunov function for $L$ \emph{with respect to $Z$ and $h$}.
\end{defi}

The following result shows that time dependent Lyapunov functions are integrable with respect to the measure $p(t,x,y)dy$ for any $(t,x)\in (0,1)\times \R^d$.
\begin{prop}\label{Prop: Time dependent Lyapunov functions are integrable with respect to pdy}
If $W$ is a time dependent Lyapunov function for $L$ with respect to $h$, then for $\xi_W (t,x):=\int_{\R^d}p(t,x,y)W(t,y)\,dy$, we have
\[
\xi_W (t,x)\le e^{\int_0^th(s)\,ds}W(0,x),\quad \forall (t,x)\in [0,1]\times \R^d.\]
\end{prop}
\begin{proof}
The proof is similar to the one given in \cite[Proposition 2.3]{Spina08}.
\end{proof}

The following results give conditions under which certain exponentials are
time dependent Lyapunov functions. Here $x\mapsto |x|_*^\beta$ denotes any $C^2$-function which coincides with $x\mapsto |x|^\beta$ for $|x|\ge 1$.

\begin{prop}\label{Prop: Time dependent Lyapunov functions in case of polynomially growing diffusions}
Assume that there is a constant $c_q>0$ such that 
\begin{equation}\label{Hp1 Prop: Time dependent Lyapunov functions in case of polynomially growing diffusions}
\sum_{i,j=1}^dq_{ij}(x)\xi_i\xi_j\le c_q(1+|x|^m)|\xi|^2
\end{equation}
holds for all $\xi ,\,x\in \R^d$ and some $m> 0$. Consider the function $W(t,x)=e^{\varepsilon t^\alpha |x|_*^\beta}$ with $\beta >(2-m)\vee 0,\,\varepsilon >0$ and $\alpha >\frac{\beta}{\beta +m-2}$. If 
\begin{equation}\label{Hp2 Prop: Time dependent Lyapunov functions in case of polynomially growing diffusions}
\limsup_{|x|\to \infty}|x|^{1-\beta -m}\left(G\cdot \frac{x}{|x|}-\frac{V}{\varepsilon \beta |x|^{\beta -1}}\right)<-\Lambda
\end{equation}
is satisfied for $\Lambda >c_q\varepsilon \beta$ 
and
\begin{equation}\label{Hp3 Prop: Time dependent Lyapunov functions in case of polynomially growing diffusions}
\lim_{\abs{x}\to\infty} V(x) \abs{x}^{2-2\beta-m}>c
\end{equation}
holds true for some $c>0$, then $W$ is a time dependent Lyapunov function for $L$ with respect to $Z(x)=e^{\varepsilon |x|_*^\beta}$ and $h(t)=C_1t^{\alpha -\gamma(2\beta+m-2)}$ for some $\gamma >\frac{1}{\beta +m-2}$ and some constant $C_1>0$. Here $G_j:=\sum_{i=1}^dD_iq_{ij}$. Moreover,
$$\xi_W (t,x)\le e^{\int_0^1 h(s)\,ds}=:C_2$$ for all $(t,x)\in [0,1]\times \R^d$.
\end{prop}

\begin{proof}
It's easy to see that $W\in C^{1,2}((0,1)\times \R^d)\cap C([0,1]\times \R^d)$, $\lim_{|x|\to \infty}W(t,x)=\infty$ uniformly for $t$ in compact subsets of $(0,1]$ and $W\le Z$. It remains to show that there is $0\le h\in L^1(0,1)$ such that \eqref{eq1: definition time dependent Lyapunov functions} and \eqref{eq2: definition time dependent Lyapunov functions} hold true.

In the following computations we assume that $\abs{x}\ge 1$ so that $\abs{x}_*^s=\abs{x}^s$ for $s\ge 0$. Otherwise, if $\abs{x}\le 1$, by continuity the functions $(t,x)\to |W(t,x)^{-1}LW(t,x)|$ and $(t,x)\to |W(t,x)^{-1}[\partial_t W(t,x)+\eta\Delta W(t,x)-V(x)W(t,x)]|$ are bounded on $(0,1)\times B_1$. Thus, we possibly choose a larger constant $C_1$ to define the function $h(t)$.

Let $t\in (0,1)$ and $\abs{x}\ge 1$. By straightforward computations we have
\begin{align*}
D_jW(t,x)=&\varepsilon\beta t^\alpha \abs{x}^{\beta-2} x_j W(t,x),\\
D_i(q_{ij}D_j W)(t,x)=& \varepsilon\beta t^\alpha \abs{x}^{\beta-2} D_iq_{ij}(x)x_j W(t,x) + \varepsilon\beta (\beta-2) t^\alpha \abs{x}^{\beta-4} q_{ij}(x) x_i x_j W(t,x) \\ 
&+ \varepsilon\beta t^\alpha \abs{x}^{\beta-2} q_{ij}(x) \delta_{ij} W(t,x)   + \varepsilon^2\beta^2 t^{2\alpha} \abs{x}^{2\beta-4} q_{ij}(x) x_i x_j W(t,x).
\end{align*}
Then, we obtain
\begin{align}\label{eq4 Prop: Time dependent Lyapunov functions in case of polynomially growing diffusions}
LW(t,x)=& \partial_t W(t,x) +AW(t,x)\notag\\
=& \varepsilon\alpha t^{\alpha-1} \abs{x}^{\beta} W(t,x)
+ \varepsilon\beta t^\alpha \abs{x}^{\beta-2} W(t,x) \sum_{i,j=1}^d D_iq_{ij}(x)x_j \notag\\
&+ \varepsilon\beta (\beta-2) t^\alpha \abs{x}^{\beta-4} W(t,x) \sum_{i,j=1}^d q_{ij}(x) x_i x_j 
+ \varepsilon\beta t^\alpha \abs{x}^{\beta-2}W(t,x) \sum_{i,j=1}^d q_{ij}(x) \delta_{ij} \notag\\
&+ \varepsilon^2\beta^2 t^{2\alpha} \abs{x}^{2\beta-4}W(t,x) \sum_{i,j=1}^d q_{ij}(x) x_i x_j -V(x)W(t,x).
\end{align}
We recall that $G_j:=\sum_{i=1}^dD_iq_{ij}$ and we use the polynomially growth of the diffusion coefficients  \eqref{Hp1 Prop: Time dependent Lyapunov functions in case of polynomially growing diffusions}. 
 We have
\begin{align*}
LW(t,x)
\le & \varepsilon\alpha t^{\alpha-1} \abs{x}^{\beta} W(t,x) + \varepsilon\beta t^\alpha \abs{x}^{\beta-2} W(t,x) G(x)\cdot x \\
&+ c_q \varepsilon\beta (\beta-2)^+ t^\alpha \abs{x}^{\beta-4} (1+\abs{x}^m)\abs{x}^2 W(t,x)  \\
&+ d c_q \varepsilon\beta t^\alpha \abs{x}^{\beta-2} (1+\abs{x}^m) W(t,x)\\
&+ c_q \varepsilon^2\beta^2 t^{2\alpha} \abs{x}^{2\beta-4} (1+\abs{x}^m)\abs{x}^2W(t,x)  -V(x)W(t,x).
\end{align*} 
 
Since $(1+\abs{x}^m)\leq 2 \abs{x}^m$ and $t^\alpha\leq 1$, we arrange the terms as follows.
\begin{align}\label{eq1 Prop: Time dependent Lyapunov functions in case of polynomially growing diffusions}
LW(t,x)
\le  \varepsilon\beta t^\alpha\abs{x}^{2\beta+m-2} W(t,x)\Bigg[& \frac{\alpha}{\beta t} \abs{x}^{2-\beta-m}+2c_q ((\beta-2)^++d) \abs{x}^{-\beta} + c_q \varepsilon\beta t^\alpha\notag\\ 
&+c_q \varepsilon\beta t^\alpha \abs{x}^{-m}
+ \abs{x}^{1-\beta-m} \left( G\cdot \frac{x}{\abs{x}} - \frac{V}{\varepsilon\beta\abs{x}^{\beta-1}}\right)\Bigg].
\end{align} 
Let $\gamma>\frac{1}{\beta+m-2}$. We distinguish two cases.

\textit{Case 1:} $\displaystyle\abs{x}>\frac{1}{t^\gamma}$.

Since $t^\alpha\le 1$ and using \eqref{eq1 Prop: Time dependent Lyapunov functions in case of polynomially growing diffusions}, we get
\begin{align}\label{eq5 Prop: Time dependent Lyapunov functions in case of polynomially growing diffusions}
LW(t,x)
\le \varepsilon\beta t^\alpha\abs{x}^{2\beta+m-2} W(t,x) \Bigg[& \frac{\alpha}{\beta} \abs{x}^{\frac{1}{\gamma}+2-\beta-m}+2c_q ((\beta-2)^++d) \abs{x}^{-\beta} + c_q \varepsilon\beta \notag\\ 
&+c_q \varepsilon\beta \abs{x}^{-m}
+ \abs{x}^{1-\beta-m} \left( G\cdot \frac{x}{\abs{x}} - \frac{V}{\varepsilon\beta\abs{x}^{\beta-1}}\right)\Bigg]. 
\end{align}
We claim that, if we assume further that $\abs{x}$ is large enough, then
\begin{equation*}
LW(t,x)\leq 0,
\end{equation*}
for all $t\in (0,1)$. To see this, let $\abs{x}>K$ for some $K>1$.
Combining \eqref{Hp2 Prop: Time dependent Lyapunov functions in case of polynomially growing diffusions} with \eqref{eq5 Prop: Time dependent Lyapunov functions in case of polynomially growing diffusions}  yields
\begin{align}\label{eq2 Prop: Time dependent Lyapunov functions in case of polynomially growing diffusions}
LW(t,x)
\le \varepsilon\beta t^\alpha\abs{x}^{2\beta+m-2} W(t,x) \bigg[& \frac{\alpha}{\beta} \abs{x}^{\frac{1}{\gamma}+2-\beta-m}+2c_q ((\beta-2)^++d) \abs{x}^{-\beta} + c_q \varepsilon\beta \notag\\ 
&+c_q \varepsilon\beta \abs{x}^{-m}
-\Lambda \bigg]. 
\end{align}
Considering that $\gamma>\frac{1}{\beta+m-2}$, $\beta>0$ and $m> 0$, we infer that
\begin{align*}
&\frac{\alpha}{\beta} \abs{x}^{\frac{1}{\gamma}+2-\beta-m}+2c_q ((\beta-2)^++d) \abs{x}^{-\beta} + c_q \varepsilon\beta \notag 
+c_q \varepsilon\beta \abs{x}^{-m}
-\Lambda\\
&\le \left(\frac{\alpha}{\beta} +2c_q ((\beta-2)^++d) + c_q \varepsilon\beta\right) K^{-l}+c_q \varepsilon\beta
-\Lambda,
\end{align*}
where $l:=\min(\frac{-1}{\gamma}-2+\beta +m, \beta, m)>0$. Since $\Lambda>c_q\varepsilon\beta$, choosing 
\begin{equation*}
K\ge \left(\frac{\frac{\alpha}{\beta} +2c_q ((\beta-2)^++d) + c_q \varepsilon\beta}{\Lambda-c_q \varepsilon\beta}\right)^\frac{1}{l},
\end{equation*}
it follows that the quantity within square brackets in the right hand side of \eqref{eq2 Prop: Time dependent Lyapunov functions in case of polynomially growing diffusions} is negative. Thus $LW(t,x)\leq 0$ for $\abs{x}>\frac{1}{t^\gamma}$, $\abs{x}>K$ and for all $t\in(0,1)$.

For the remaining values of $x$, $\abs{x}\leq K$, we have that $LW(t,x)\leq C$ for a certain constant $C>0$.
Anyway, we conclude that
\begin{equation*}
LW(t,x)\leq CW(t,x),
\end{equation*}
for all $ t\in (0,1)$ and $\abs{x}>\frac{1}{t^\gamma}$.

\textit{Case 2:} $\displaystyle\abs{x}\leq \frac{1}{t^\gamma}$.

We assume that $\abs{x}$ is large enough. Otherwise, as in Case 1, $LW(t,x)\leq C\leq CW(t,x)$ for a certain constant $C$.
We combine \eqref{Hp2 Prop: Time dependent Lyapunov functions in case of polynomially growing diffusions} and \eqref{eq1 Prop: Time dependent Lyapunov functions in case of polynomially growing diffusions} to deduce that
\begin{align*}
LW(t,x)\leq &\big[\varepsilon\alpha t^{\alpha-1-\gamma\beta}+2c_q \varepsilon\beta((\beta-2)^++d)t^{\alpha-\gamma(\beta+m-2)}+c_q \varepsilon^2 \beta^2 t^{2\alpha-\gamma(2\beta+m-2)}\\
&+c_q \varepsilon^2 \beta^2 t^{2\alpha-\gamma(2\beta-2)}-\varepsilon\beta t^\alpha \abs{x}^{2\beta+m-2}\Lambda\big]W(t,x).
\end{align*}
We drop the term involving $\Lambda$ because it's negative. Moreover, since $\gamma>1/(\beta+m-2)$, we note that the leading term is $t^{\alpha-\gamma(2\beta+m-2)}$. Hence
\begin{equation*}
LW(t,x)\leq h(t)W(t,x),
\end{equation*}  
where
\begin{equation*}
h(t):=C_1 t^{\alpha-\gamma(2\beta+m-2)}.
\end{equation*}
For the function $h(t)$ to be in the space $L^1((0,1))$, we set $\alpha>\frac{\beta}{\beta+m-2}$. In this way, choosing $\gamma<\frac{\alpha+1}{2\beta+m-2}$ so that $\alpha-\gamma(2\beta+m-2)>-1$, $h(t)$ is integrable in the interval $(0,1)$.

Summing up, considering a possibly larger constant $C_1$, we proved \eqref{eq1: definition time dependent Lyapunov functions} for all $t\in (0,1)$ and $x\in\R^d$.

We now verify \eqref{eq2: definition time dependent Lyapunov functions}. An easy computation shows that
\begin{equation*}
\Delta W(t,x)
= \varepsilon\beta (\beta+d-2) t^\alpha \abs{x}^{\beta-2} W(t,x) +\varepsilon^2\beta^2 t^{2\alpha} \abs{x}^{2\beta-2} W(t,x). 
\end{equation*}
Thus, we get
\begin{align}\label{eq3 Prop: Time dependent Lyapunov functions in case of polynomially growing diffusions}
\partial_t W(t,x)+\eta\Delta W(t,x)-V(x)W(t,x)
= \, & \varepsilon\alpha t^{\alpha-1} \abs{x}^{\beta} W(t,x) \notag \\
&+ \eta \varepsilon\beta (\beta+d-2) t^\alpha \abs{x}^{\beta-2} W(t,x)\notag \\
&+ \eta \varepsilon^2\beta^2 t^{2\alpha} \abs{x}^{2\beta-2} W(t,x) -V(x)W(t,x).
\end{align}
As in the first part of the proof, we let $\gamma>\frac{1}{\beta+m-2}$ and we distinguish two cases.

\textit{Case 1:} $\displaystyle\abs{x}>\frac{1}{t^\gamma}$.

Since $t^\alpha\leq 1$, by \eqref{eq3 Prop: Time dependent Lyapunov functions in case of polynomially growing diffusions} we obtain
\begin{align*}
 \partial_t  W(t,x)+\eta\Delta W(t,x)-V&(x)W(t,x)\notag\\
 \leq \varepsilon\beta t^\alpha\abs{x}^{2\beta+m-2} W(t,x) \Bigg[ & \frac{\alpha}{\beta} \abs{x}^{\frac{1}{\gamma}+2-\beta-m} 
+\eta(\beta+d-2) \abs{x}^{-\beta-m}  \notag \\
&+ \eta \varepsilon\beta \abs{x}^{-m}
- \frac{1}{\varepsilon\beta}V(x)\abs{x}^{2-2\beta-m}\Bigg].  
\end{align*}
If $\abs{x}$ large enough, by \eqref{Hp3 Prop: Time dependent Lyapunov functions in case of polynomially growing diffusions} we have
\begin{align*}
& \partial_t  W(t,x)+\eta\Delta W(t,x)-V(x)W(t,x)\notag\\
& \leq \varepsilon\beta t^\alpha\abs{x}^{2\beta+m-2} W(t,x) \Bigg[  \frac{\alpha}{\beta} \abs{x}^{\frac{1}{\gamma}+2-\beta-m} 
+\eta(\beta+d-2) \abs{x}^{-\beta-m} + \eta \varepsilon\beta \abs{x}^{-m}
- \frac{c}{\varepsilon \beta} \Bigg].
\end{align*}
Arguing as in \eqref{eq2 Prop: Time dependent Lyapunov functions in case of polynomially growing diffusions}, we find that $\partial_t  W(t,x)+\eta\Delta W(t,x)-V(x)W(t,x)$ is negative for $\abs{x}$ large, whereas it's bounded for the remaining values of $x$. Therefore, we deduce that
\begin{equation*}
\partial_t  W(t,x)+\eta\Delta W(t,x)-V(x)W(t,x)\leq C W(t,x),
\end{equation*}
for all $ t\in (0,1)$ and $\abs{x}>\frac{1}{t^\gamma}$.

\textit{Case 2:} $\displaystyle\abs{x}\leq \frac{1}{t^\gamma}$.

Since $V\geq 0$, \eqref{eq3 Prop: Time dependent Lyapunov functions in case of polynomially growing diffusions} leads to
\begin{align*}
\partial_t  W(t,x)+\eta\Delta W(t,x)-V(x)W(t,x)
\leq \big[&\varepsilon\alpha t^{\alpha-1-\gamma\beta}
+\eta \varepsilon\beta((\beta -2)^++d)t^{\alpha-\gamma(\beta-2)}\\
&+\eta\varepsilon^2 \beta^2 t^{2\alpha-\gamma(2\beta-2)}\big]W(t,x).
\end{align*}
We can control the right hand side of the previous inequality with the function $h(t)W(t,x)$, obtaining that
\begin{equation*}
\partial_t  W(t,x)+\eta\Delta W(t,x)-V(x)W(t,x)\leq h(t)W(t,x),
\end{equation*}
where the constant $C_1$ in the function $h$ has to be suitably adjusted.
In both cases \eqref{eq2: definition time dependent Lyapunov functions} holds true.
We conclude that $W$ is a time dependent Lyapunov function for $L$.

Moreover, by Proposition \ref{Prop: Time dependent Lyapunov functions are integrable with respect to pdy}, we have 
$$\xi_W (t,x)\le e^{\int_0^t h(s)\,ds} W(0,x) \le e^{\int_0^1 h(s)\,ds}=:C_2$$ for all $(t,x)\in [0,1]\times \R^d$.
\end{proof}

\begin{rem}
One can easily see that the same conclusion as in Proposition \ref{Prop: Time dependent Lyapunov functions in case of polynomially growing diffusions} remains valid if we replace the operator $A$ with the more general operator $A_F:=A+F\cdot \nabla$ with $F\in C^\zeta(\R^d, \R^d)$ for some $\zeta\in (0,1)$, and the condition 
\eqref{Hp2 Prop: Time dependent Lyapunov functions in case of polynomially growing diffusions} with
$$\limsup_{|x|\to \infty}|x|^{1-\beta -m}\left((G+F)\cdot \frac{x}{|x|}-\frac{V}{\varepsilon \beta |x|^{\beta -1}}\right)<-\Lambda .$$ This generalizes Proposition 2.3 in \cite{ALR10}.
\end{rem}

Let us now consider the case of exponentially growing coefficients.

\begin{prop}\label{Prop: Time dependent Lyapunov functions in case of exponentially growing diffusions}
Assume that there is a constant $c_e>0$ such that 
\begin{equation}\label{Hp1 Prop: Time dependent Lyapunov functions in case of exponentially growing diffusions}
\sum_{i,j=1}^dq_{ij}(x)\xi_i\xi_j\le c_e e^{|x|^m}|\xi|^2
\end{equation}
holds for all $\xi ,\,x\in \R^d$ and some $m\ge 2$. Consider the function 
\begin{equation*}
W(t,x)=\exp\bigg(\varepsilon t^\alpha \int_0^{\abs{x}_*} e^{\frac{\tau^\beta}{2}}d\tau\bigg)
\end{equation*}
 with $\frac{m}{2}+1\le \beta \leq m,\,\varepsilon >0$ and $\alpha >\frac{2\beta+m-2}{2m}$. If 
\begin{equation}\label{Hp2 Prop: Time dependent Lyapunov functions in case of exponentially growing diffusions}
\limsup_{|x|\to \infty}|x|^{1-\beta-m}e^{-\frac{\abs{x}^\beta}{2}-\abs{x}^m}\left(G\cdot \frac{x}{|x|}-\frac{V}{\varepsilon e^{\frac{\abs{x}^\beta}{2}}}\right)<-\Lambda
\end{equation}
is satisfied for $\Lambda >0$ and
\begin{equation}\label{Hp3 Prop: Time dependent Lyapunov functions in case of exponentially growing diffusions}
\lim_{\abs{x}\to\infty} V(x) \abs{x}^{1-\beta-m} e^{-\abs{x}^\beta-\abs{x}^m}>c
\end{equation}
holds true for some $c>0$, then $W$ is a time dependent Lyapunov function for $L$ with respect to $Z(x)=\exp\big(\varepsilon \int_0^{\abs{x}_*} e^{\frac{\tau^\beta}{2}}d\tau\big)$ and $h(t)=C_3 t^{\alpha-\gamma \left( \beta+\frac{3}{2}m-1\right)}$ for some $\gamma >\frac{1}{m}$ and some constant $C_3>0$. Here $G_j:=\sum_{i=1}^dD_iq_{ij}$. Moreover,
$$\xi_W (t,x)\le e^{\int_0^1 h(s)\,ds}=:C_4$$ for all $(t,x)\in [0,1]\times \R^d$.
\end{prop}

\begin{proof}
Throughout the proof we assume that $\abs{x}\geq 1$ so that $\abs{x}_*^s=\abs{x}^s$ for $s\ge 0$. The estimates can be extended to $\R^d$ by possibly choosing larger constants.

Let $t\in (0,1)$ and $\abs{x}\ge 1$. By direct computations we have
\begin{align*}
D_jW(t,x)=&\varepsilon t^\alpha \frac{x_j}{\abs{x}} e^\frac{\abs{x}^\beta}{2} W(t,x),\\
D_i(q_{ij}D_j W)(t,x)
=& \varepsilon t^\alpha \frac{1}{\abs{x}} e^\frac{\abs{x}^\beta}{2} D_iq_{ij}(x)x_j W(t,x) 
+ \frac{1}{2}\varepsilon\beta t^\alpha \abs{x}^{\beta-3} e^\frac{\abs{x}^\beta}{2} q_{ij}(x) x_i x_j W(t,x) \\ 
&+ \varepsilon t^\alpha \frac{1}{\abs{x}} e^\frac{\abs{x}^\beta}{2} q_{ij}(x) \delta_{ij} W(t,x)
- \varepsilon t^{\alpha} \frac{1}{\abs{x}^3} e^\frac{\abs{x}^\beta}{2} q_{ij}(x) x_i x_j W(t,x)\\
&+ \varepsilon^2 t^{2\alpha}\frac{1}{\abs{x}^2} e^{\abs{x}^\beta} q_{ij}(x) x_i x_j W(t,x).
\end{align*}
Hence we deduce that
\begin{align*}
LW(t,x)=& \partial_t W(t,x) +AW(t,x)\\
=& \varepsilon\alpha t^{\alpha-1} W(t,x) \int_0^{\abs{x}} e^{\frac{\tau^\beta}{2}}d\tau
+ \varepsilon t^\alpha \frac{1}{\abs{x}} e^\frac{\abs{x}^\beta}{2} W(t,x) \sum_{i,j=1}^d D_iq_{ij}(x)x_j \\
&+ \frac{1}{2}\varepsilon\beta t^\alpha \abs{x}^{\beta-3} e^\frac{\abs{x}^\beta}{2} W(t,x) \sum_{i,j=1}^d q_{ij}(x) x_i x_j 
+ \varepsilon t^\alpha \frac{1}{\abs{x}} e^\frac{\abs{x}^\beta}{2} W(t,x) \sum_{i,j=1}^d q_{ij}(x) \delta_{ij} \\
&- \varepsilon t^{\alpha} \frac{1}{\abs{x}^3} e^\frac{\abs{x}^\beta}{2} W(t,x) \sum_{i,j=1}^d q_{ij}(x) x_i x_j 
+ \varepsilon^2 t^{2\alpha}\frac{1}{\abs{x}^2} e^{\abs{x}^\beta} W(t,x) \sum_{i,j=1}^d q_{ij}(x) x_i x_j \\
&-V(x)W(t,x).
\end{align*}
First of all, we drop the negative term involving $\varepsilon$ in the right hand side of the previous equality. Second, we use the exponentially growth of the diffusion coefficients \eqref{Hp1 Prop: Time dependent Lyapunov functions in case of exponentially growing diffusions} to obtain that
\begin{align}\label{eq1 Prop: Time dependent Lyapunov functions in case of exponentially growing diffusions}
LW(t,x)
\leq & \varepsilon\alpha t^{\alpha-1} W(t,x) \int_0^{\abs{x}} e^{\frac{\tau^\beta}{2}}d\tau
+ \varepsilon t^\alpha \frac{1}{\abs{x}} e^\frac{\abs{x}^\beta}{2} W(t,x) G(x)\cdot x \notag\\
&+ \frac{1}{2}c_e\varepsilon\beta t^\alpha \abs{x}^{\beta-1} e^{\frac{\abs{x}^\beta}{2}+\abs{x}^m} W(t,x) 
+ d c_e\varepsilon t^\alpha \frac{1}{\abs{x}} e^{\frac{\abs{x}^\beta}{2}+\abs{x}^m} W(t,x)  \notag\\
&+ c_e\varepsilon^2 t^{2\alpha} e^{\abs{x}^\beta+\abs{x}^m} W(t,x) -V(x)W(t,x).
\end{align}
Since $t^\alpha\leq 1$, we can write the previous inequality as follows:
\begin{align}\label{eq2 Prop: Time dependent Lyapunov functions in case of exponentially growing diffusions}
LW(t,x)
\leq  \varepsilon t^\alpha\abs{x}^{\beta+m-1} e^{\abs{x}^\beta+\abs{x}^m} W(t,x)
\Bigg[& \frac{\alpha}{t} \abs{x}^{1-\beta-m}e^{-\abs{x}^\beta-\abs{x}^m} \int_0^{\abs{x}} e^{\frac{\tau^\beta}{2}}d\tau
+\frac{1}{2}c_e \beta \abs{x}^{-m} e^{-\frac{\abs{x}^\beta}{2}}\notag\\
&+ d c_e \abs{x}^{-\beta-m}e^{-\frac{\abs{x}^\beta}{2}} 
+c_e\varepsilon t^\alpha \abs{x}^{1-\beta-m}\notag\\
&+ \abs{x}^{1-\beta-m} e^{-\frac{\abs{x}^\beta}{2}-\abs{x}^m}\left( G\cdot \frac{x}{\abs{x}} - \frac{V}{\varepsilon e^{\frac{\abs{x}^\beta}{2}}}\right)\Bigg].
\end{align} 

\noindent Let $\gamma>\frac{1}{m}$. We now distinguish two cases.

\textit{Case 1:} $\displaystyle e^{\abs{x}^m}\geq\frac{1}{t^{\gamma m}}$.

First, we observe that
$$\int_0^{\abs{x}} e^{\frac{\tau^\beta}{2}}d\tau\leq \abs{x}e^{\frac{\abs{x}^\beta}{2}}. $$
Then, since $t^\alpha\leq 1$ and $e^{-\frac{\abs{x}^\beta}{2}}\leq 1$, by \eqref{eq2 Prop: Time dependent Lyapunov functions in case of exponentially growing diffusions} we get
\begin{align*}
LW(t,x)
\leq  \varepsilon t^\alpha\abs{x}^{\beta+m-1} e^{\abs{x}^\beta+\abs{x}^m} W(t,x)
\Bigg[& \alpha \abs{x}^{2-\beta-m}e^{\left( \frac{1}{\gamma m}-1\right)\abs{x}^m} 
+\frac{1}{2}c_e \beta \abs{x}^{-m}\\
&+ d c_e \abs{x}^{-\beta-m}
+c_e\varepsilon \abs{x}^{1-\beta-m}\\
&+ \abs{x}^{1-\beta-m} e^{-\frac{\abs{x}^\beta}{2}-\abs{x}^m}\left( G\cdot \frac{x}{\abs{x}} - \frac{V}{\varepsilon e^{\frac{\abs{x}^\beta}{2}}}\right)\Bigg].
\end{align*} 
Moreover, $e^{\left( \frac{1}{\gamma m}-1\right)\abs{x}^m}\leq 1$ because $\gamma>\frac{1}{m}$. Thus, we derive that
\begin{align*}
LW(t,x)
\leq  \varepsilon t^\alpha\abs{x}^{\beta+m-1} e^{\abs{x}^\beta+\abs{x}^m} W(t,x)
\Bigg[& \alpha \abs{x}^{2-\beta-m}
+\frac{1}{2}c_e \beta \abs{x}^{-m}\\
&+ d c_e \abs{x}^{-\beta-m}
+c_e\varepsilon \abs{x}^{1-\beta-m}\\
&+ \abs{x}^{1-\beta-m} e^{-\frac{\abs{x}^\beta}{2}-\abs{x}^m}\left( G\cdot \frac{x}{\abs{x}} - \frac{V}{\varepsilon e^{\frac{\abs{x}^\beta}{2}}}\right)\Bigg].
\end{align*} 
If $\abs{x}$ is large enough, say $\abs{x}>K$ for some $K>1$, we apply \eqref{Hp2 Prop: Time dependent Lyapunov functions in case of exponentially growing diffusions} to deduce that
\begin{align*}
LW(t,x)
\leq  \varepsilon t^\alpha\abs{x}^{\beta+m-1} e^{\abs{x}^\beta+\abs{x}^m} W(t,x)
\Bigg[& \alpha \abs{x}^{2-\beta-m}
+\frac{1}{2}c_e \beta \abs{x}^{-m}+ d c_e \abs{x}^{-\beta-m}\\
&+c_e\varepsilon \abs{x}^{1-\beta-m}-\Lambda\Bigg].
\end{align*} 
We now show that, for a suitable choice of $K$, the quantity within square brackets is negative. 
Since $\beta \ge \frac{m}{2}+1$ and $m\ge 2$, we have $\beta \ge 2$ and hence
%
\begin{align*}
&\alpha \abs{x}^{2-\beta-m}
+\frac{1}{2}c_e \beta \abs{x}^{-m}+ d c_e \abs{x}^{-\beta-m}
+c_e\varepsilon \abs{x}^{1-\beta-m}-\Lambda\\
&\leq  \left(\alpha +\frac{1}{2}c_e \beta + d c_e  +c_e\varepsilon \right) K^{-m}-\Lambda.
\end{align*} 
As a result, by taking
$$K\geq \bigg(\frac{\alpha +\frac{1}{2}c_e \beta + d c_e  +c_e\varepsilon }{\Lambda}\bigg)^{\frac{1}{m}},$$
we finally get $LW(t,x)\leq 0$. For the remaining values of $x$, $LW$ is bounded by a constant. In both cases we have
\begin{equation*}
LW(t,x)\leq CW(t,x),
\end{equation*}
for all $t\in (0,1)$, $e^{\abs{x}^m}\geq\frac{1}{t^{\gamma m}}$ and for some constant $C>0$.

\textit{Case 2:} $\displaystyle e^{\abs{x}^m}<\frac{1}{t^{\gamma m}}$.

Notice that $\abs{x}<t^{-\gamma}$ and, since $\beta\leq m$, we have
$$e^{\abs{x}^\beta}<\frac{1}{t^{\gamma m}} \,\hbox{\ for }\abs{x} \ge 1.$$
Then, if $\abs{x}$ is large enough, using $\beta >1$, and combining \eqref{Hp2 Prop: Time dependent Lyapunov functions in case of exponentially growing diffusions} and \eqref{eq2 Prop: Time dependent Lyapunov functions in case of exponentially growing diffusions}, we obtain that
\begin{align*}
LW(t,x)
\leq \Bigg[ &\varepsilon \alpha t^{\alpha-1-\gamma \left(\frac{m}{2}+1\right)}+\frac{1}{2}c_e \varepsilon \beta t^{\alpha-\gamma \left( \beta+\frac{3}{2}m-1\right)}+dc_e\varepsilon t^{\alpha-\frac{3}{2}\gamma m}+c_e \varepsilon^2 t^{2\alpha-2\gamma m}\\
&-\Lambda \varepsilon t^\alpha \abs{x}^{\beta+m-1} e^{\abs{x}^\beta+\abs{x}^m}\Bigg]W(t,x).
\end{align*}
Dropping the last negative term, we find
\begin{align*}
LW(t,x)
\leq \Bigg[ &\varepsilon \alpha t^{\alpha-1-\gamma \left(\frac{m}{2}+1\right)}+\frac{1}{2}c_e \varepsilon \beta t^{\alpha-\gamma \left( \beta+\frac{3}{2}m-1\right)}+dc_e\varepsilon t^{\alpha-\frac{3}{2}\gamma m}+c_e \varepsilon^2 t^{2\alpha-2\gamma m}\Bigg]W(t,x).
\end{align*}
Since $\gamma>\frac{1}{m}$ and $\beta\ge \frac{m}{2}+1$, the leading term is $t^{\alpha-\gamma \left( \beta+\frac{3}{2}m-1\right)}$. Therefore, we gain
$$LW(t,x)\leq C t^{\alpha-\gamma \left( \beta+\frac{3}{2}m-1\right)}W(t,x),$$
for all $t\in (0,1)$, $e^{\abs{x}^m}<\frac{1}{t^{\gamma m}}$ and for some constant $C>0$.

To sum up, there exists a constant $C_3>0$ such that
$$LW(t,x)\leq h(t) W(t,x),$$
for all $t\in (0,1)$ and $x\in\R^d$, where $h(t)=C_3 t^{\alpha-\gamma \left( \beta+\frac{3}{2}m-1\right)}$.

Moreover, we choose $\gamma<\frac{\alpha+1}{\beta+\frac{3}{2}m-1}$, which is possible since $\alpha >\frac{2\beta+m-2}{2m}$, so that $\alpha-\gamma \left( \beta+\frac{3}{2}m-1\right)>-1$ and $h\in L^1((0,1))$.
We conclude that condition \eqref{eq1: definition time dependent Lyapunov functions} is satisfied.

To show \eqref{eq2: definition time dependent Lyapunov functions} we compute
\begin{align*}
\Delta W(t,x)
=& \frac{1}{2} \varepsilon\beta t^\alpha \abs{x}^{\beta-1}e^{\frac{\abs{x}^\beta}{2}}W(t,x) 
+d\varepsilon t^\alpha \frac{1}{\abs{x}} e^\frac{\abs{x}^\beta}{2} W(t,x)
-\varepsilon t^\alpha \frac{1}{\abs{x}} e^\frac{\abs{x}^\beta}{2} W(t,x)\\
&+\varepsilon^2 t^{2\alpha} e^{\abs{x}^\beta} W(t,x).
\end{align*}
Hence,
\begin{align}\label{eq3 Prop: Time dependent Lyapunov functions in case of exponentially growing diffusions}
\partial_t W(t,x)+\eta \Delta W(t,x)-V(x)W(t,x)
=& \varepsilon\alpha t^{\alpha-1}W(t,x) \int_0^{\abs{x}} e^{\frac{\tau^\beta}{2}}d\tau + \frac{1}{2} \eta\varepsilon\beta t^\alpha \abs{x}^{\beta-1}e^{\frac{\abs{x}^\beta}{2}}W(t,x) \notag\\
&+d\eta\varepsilon t^\alpha \frac{1}{\abs{x}} e^\frac{\abs{x}^\beta}{2} W(t,x)
-\eta\varepsilon t^\alpha \frac{1}{\abs{x}} e^\frac{\abs{x}^\beta}{2} W(t,x)\notag\\
&+\eta\varepsilon^2 t^{2\alpha} e^{\abs{x}^\beta} W(t,x)-V(x)W(t,x)\notag\\
\leq & \varepsilon\alpha t^{\alpha-1}W(t,x) \int_0^{\abs{x}} e^{\frac{\tau^\beta}{2}}d\tau + \frac{1}{2} \eta\varepsilon\beta t^\alpha \abs{x}^{\beta-1}e^{\frac{\abs{x}^\beta}{2}}W(t,x) \notag\\
&+d\eta\varepsilon t^\alpha \frac{1}{\abs{x}} e^\frac{\abs{x}^\beta}{2} W(t,x)
+\eta\varepsilon^2 t^{2\alpha} e^{\abs{x}^\beta} W(t,x)\notag\\
&-V(x)W(t,x).
\end{align}
We use the same strategy as above. We let $\gamma>\frac{1}{m}$ and we consider two cases.

\textit{Case 1:} $\displaystyle e^{\abs{x}^m}\geq\frac{1}{t^{\gamma m}}$.

By \eqref{eq3 Prop: Time dependent Lyapunov functions in case of exponentially growing diffusions} we obtain
\begin{align*}
\partial_t W(t,x)+\eta \Delta W(t,x)-V(x)W(t,x)
\leq &\varepsilon t^\alpha \abs{x}^{\beta+m-1} e^{\abs{x}^\beta+\abs{x}^m}W(t,x)\Bigg[\alpha \abs{x}^{2-\beta-m} e^{\left(\frac{1}{\gamma m}-1\right) \abs{x}^m}\\
& +\frac{1}{2}\eta\beta\abs{x}^{-m}+d\eta \abs{x}^{-\beta-m}+\eta\varepsilon \abs{x}^{1-\beta-m}\\
&-\frac{1}{\varepsilon} V(x) \abs{x}^{1-\beta-m} e^{-\abs{x}^\beta-\abs{x}^m}\Bigg].
\end{align*}
Using \eqref{Hp3 Prop: Time dependent Lyapunov functions in case of exponentially growing diffusions} and the fact that $\gamma>\frac{1}{m}$, we get
\begin{align*}
\partial_t W(t,x)+\eta \Delta W(t,x)-V(x)W(t,x)
\leq &\varepsilon t^\alpha \abs{x}^{\beta+m-1} e^{\abs{x}^\beta+\abs{x}^m}W(t,x)\Big[\alpha \abs{x}^{2-\beta-m} \\
& +\frac{1}{2}\eta\beta\abs{x}^{-m}+d\eta \abs{x}^{-\beta-m}+\eta\varepsilon \abs{x}^{1-\beta-m}-\frac{c}{\varepsilon}\Big].
\end{align*}
If $\abs{x}$ is large enough, the quantity within square brackets is negative. Otherwise, we can control it with a constant. In both cases, we deduce that
$$\partial_t W(t,x)+\eta \Delta W(t,x)-V(x)W(t,x)\leq C,$$
for all $t\in (0,1)$, $e^{\abs{x}^m}\geq\frac{1}{t^{\gamma m}}$ and for some constant $C>0$.

\textit{Case 2:} $\displaystyle e^{\abs{x}^m}<\frac{1}{t^{\gamma m}}$.

Since $\beta\leq m$ and $V\geq 0$, \eqref{eq3 Prop: Time dependent Lyapunov functions in case of exponentially growing diffusions} yields
\begin{align*}
\partial_t W(t,x)+\eta \Delta W(t,x)-V(x)W(t,x)
\leq &\Big[\varepsilon\alpha t^{\alpha-1-\gamma \left(\frac{m}{2}+1\right)}
+\frac{1}{2}\eta\beta t^{\alpha-\gamma \left( \beta+\frac{m}{2}-1\right)}\\
&+d\eta \varepsilon t^{\alpha-\gamma \frac{m}{2}}+\eta\varepsilon^2 t^{2\alpha-\gamma m}\Big]W(t,x)\\
\leq & C t^{\alpha-\gamma \left( \beta+\frac{3}{2}m-1\right)} W(t,x),
\end{align*}
for some constant $C$.
Therefore, by possibly choosing a larger $C_3$, we gain \eqref{eq2: definition time dependent Lyapunov functions}. Then, $W$ is a time dependent Lyapunov function for $L$. The last assertion follows from Proposition \ref{Prop: Time dependent Lyapunov functions are integrable with respect to pdy}.
\end{proof}

\section{Kernel estimates and spectral properties for general Schr\"odinger type operators}\label{S4}

In this section we establish pointwise upper bounds for the kernel $p$ and study some spectral
properties of $A_{\min}$ with either polynomial or exponential coefficients.

To obtain pointwise kernel estimates one needs the following assumptions.

\begin{hyp}\label{h.2}
Fix $0<t\le 1,\,x\in \R^d$ and $0<a_0<a<b<b_0<t$. Let us consider two time dependent Lyapunov functions $W_1,\,W_2$ with $W_1\le W_2$ and a weight function $1\le w\in C^{1,2}((0,t)\times \R^d)$ such that there exist $k>d+2$ and constants $c_1,\ldots ,c_5$ with
\begin{eqnarray*}
& & w\le c_1w^{\frac{k-2}{k}}W_1^{\frac{2}{k}},\,|Q\nabla w|\le c_2w^{\frac{k-1}{k}}W_1^{\frac{1}{k}},\,|{\div}(Q\nabla w)|\le c_3w^{\frac{k-2}{k}}W_1^{\frac{2}{k}},\\
& & |\partial_t w|\le c_4w^{\frac{k-2}{k}}W_1^{\frac{2}{k}},\,V^{\frac{1}{2}}\le c_5w^{-\frac{1}{k}}W_2^{\frac{1}{k}}
\end{eqnarray*}
on $[a_0,b_0]\times \R^d$.
\end{hyp}

The following result can be deduced as in \cite[Theorem 12.4]{KunzeLorenziRhandi1} and \cite[Theorem 4.2]{KunzeLorenziRhandi2}. 

\begin{theo}\label{theo-estimate bounded case}
Assume Hypotheses \ref{h.1}, \ref{h.2}, $k>d+2$ and $q_{ij},\,D_kq_{ij}$ are bounded on $\R^d$. Then there is a constant $C>0$ depending only on  $d,\,k$ and $\eta$ such that
\begin{align}\label{Estimate for wp in the bounded case}
w(t,y)p(t,x,y) \le & C\left[c_1^{\frac{k}{2}}\sup_{s\in (a_0,b_0)}\xi_{W_1}(s,x)+\left(
c_2^k+\frac{c_1^{\frac{k}{2}}}{(b_0-b)^{\frac{k}{2}}}+c_3^{\frac{k}{2}}
+c_4^{\frac{k}{2}}\right)\int_{a_0}^{b_0}\xi_{W_1}(s,x)\,ds\right.\notag\\
& \quad \quad \left. +c_5^k\int_{a_0}^{b_0}\xi_{W_2}(s,x)\,ds\right],
\end{align}
for all $(t,y)\in (a,b)\times \R^d$ and any fixed $x\in \R^d$.
\end{theo}

Notice that the assumption of bounded diffusion coefficients was crucial to apply \cite[Theorem 3.7]{KunzeLorenziRhandi2}.
The fact that the constant $C$ does not depend on $\norm{Q}_{\infty}$ will allow us to extend this result to the general case.

By an approximation argument one can extend the above result to  the case of unbounded diffusion coefficients. The proof of the following result is similar to the one in \cite[Theorem 12.6]{KunzeLorenziRhandi1}. The only difference is that here we are concerned with autonomous problems. This is the reason why we assume \eqref{Further assumption2} for a fixed $t_0\in (0,t)$, compare with \cite[Hypothesis 12.5]{KunzeLorenziRhandi1}.

\begin{theo}\label{theo-estimate}
In addition to Hypotheses \ref{h.1}, \ref{h.2} and $k>d+2$, we assume that 
\begin{enumerate}
\item on $[a_0,b_0]\times \R^d$ we have
\begin{equation}\label{Further assumption1}
|\Delta w|\le c_6w^{\frac{k-2}{k}}W_1^{\frac{2}{k}};
\end{equation}
\item there is $t_0\in (0,t)$ such that
\begin{equation}\label{Further assumption2}
|Q\nabla W_1(t_0, \cdot)|\le c_7W_1(t_0, \cdot)w^{-\frac{1}{k}}W_2^{\frac{1}{k}};
\end{equation}
\item there are $c_0>0$ and $\sigma\in (0,1)$ such that
\begin{equation}\label{Further assumption3}
W_2\leq c_0 Z^{1-\sigma}.
\end{equation}
\end{enumerate}
Then there is a constant $C>0$ depending only on  $d,\,k$ and $\eta$ such that
\begin{eqnarray}\label{Estimate for wp in the general case}
w(t,y)p(t,x,y) &\le & C\left[c_1^{\frac{k}{2}}\sup_{s\in (a_0,b_0)}\xi_{W_1}(s,x)+\left(
c_2^k+\frac{c_1^{\frac{k}{2}}}{(b_0-b)^{\frac{k}{2}}}+c_3^{\frac{k}{2}}
+c_4^{\frac{k}{2}}+c_6^{\frac{k}{2}}\right)\int_{a_0}^{b_0}\xi_{W_1}(s,x)\,ds\right.\notag\\
& & \quad \quad \left. +(c_5^k+c_2^\frac{k}{2}c_7^\frac{k}{2})\int_{a_0}^{b_0}\xi_{W_2}(s,x)\,ds\right],
\end{eqnarray}
for all $(t,y)\in (a,b)\times \R^d$ and fixed $x\in \R^d$.
\end{theo}

In the following subsections we apply Theorem \ref{theo-estimate} to obtain explicit kernel estimates in the case of polynomially or exponentially coefficients. Moreover we prove in these cases the compactness of the semigroups and deduce estimates of the eigenfuncions.

\subsection{Polynomially growing coefficients}\label{S51}

\smallskip\noindent
Here we apply the results of the previous sections to the case of operators with polynomial diffusion coefficients and potential terms.

Consider $Q(x)=(1+|x|^m_*)I$ and $V(x)=|x|^s$ with $s>|m-2|$ and $m>0$. To apply Theorem \ref{theo-estimate} we set
\[
w(t,x)=e^{\varepsilon t^\alpha |x|_*^\beta} \hbox{\ and } W_j(t,x)=e^{\varepsilon_j t^\alpha |x|_*^\beta},
\]
where $j=1,2$, $\beta=\frac{s-m+2}{2},\,0<\varepsilon<\varepsilon_1<\varepsilon_2<\frac{1}{\beta}$ and $\alpha>\frac{\beta}{\beta+m-2}$.

\begin{theo}\label{Thm: Kernel estimates in case of polynomially growing diffusions}
Let $p$ be the integral kernel associated with the operator $A$ with $Q(x)=(1+|x|^m_*)I$ and $V(x)=|x|^s$, where $s>|m-2|$ and $m> 0$. Then
$$p(t,x,y)\le t^{1-\frac{\alpha(2m\vee s)}{s-m+2}k} e^{-\frac{\varepsilon}{2} t^\alpha |x|_*^{\frac{s-m+2}{2}}}
e^{-\frac{\varepsilon}{2} t^\alpha |y|_*^{\frac{s-m+2}{2}}}$$
for $k>d+2$ and any $t\in (0,1),\,x,y\in \R^d$.
\end{theo}
\begin{proof}
\textit{Step 1.} We apply Proposition \ref{Prop: Time dependent Lyapunov functions in case of polynomially growing diffusions} to verify that the operator $A$ satisfies Hypothesis \ref{h.1} with 
\[Z(x)=e^{\varepsilon_2 |x|_*^\beta}\]
and that $W_1$ and $W_2$ are time dependent Lyapunov functions for $L=\partial_t+A$.
Clearly, \eqref{Hp1 Prop: Time dependent Lyapunov functions in case of polynomially growing diffusions} holds true with $c_q=1$. Since $s>|m-2|$, we have $\beta>(2-m)\vee 0$. 
It remains to check \eqref{Hp2 Prop: Time dependent Lyapunov functions in case of polynomially growing diffusions}  and \eqref{Hp3 Prop: Time dependent Lyapunov functions in case of polynomially growing diffusions}. Let $|x|\geq 1$ and set 
$G_j=\sum_{i=1}^d D_i q_{ij} = m |x|^{m-2}x_j.$
Then
\begin{align*}
|x|^{1-\beta-m} \left(G\cdot \frac{x}{|x|}-\frac{V}{\varepsilon_j \beta |x|^{\beta -1}}\right)
=  |x|^{1-\beta-m} \left(m |x|^{m-1}- \frac{|x|^s}{\varepsilon_j \beta |x|^{\beta -1}}\right)
= m|x|^{-\beta} - \frac{1}{\varepsilon_j \beta}.
\end{align*}
If $|x|$ is large enough, for example $|x|\ge K$ with
$$K>\left( \frac{m}{\frac{1}{\varepsilon_j \beta}-1}\right)^\frac{1}{\beta}, $$
we get
\begin{align*}
|x|^{1-\beta-m} \left(G\cdot \frac{x}{|x|}-\frac{V}{\varepsilon_j \beta |x|^{\beta -1}}\right)
= m|x|^{-\beta} - \frac{1}{\varepsilon_j \beta}
\leq m K^{-\beta}- \frac{1}{\varepsilon_j \beta}
<   -1,
\end{align*}
where we have used that $\varepsilon_j <\frac{1}{\beta}$.
Hence, \eqref{Hp2 Prop: Time dependent Lyapunov functions in case of polynomially growing diffusions}  is satisfied if we choose $\Lambda :=1$.
Moreover, we have
$$\lim_{\abs{x}\to\infty} V(x) \abs{x}^{2-2\beta-m}=\lim_{\abs{x}\to\infty} \abs{x}^{2-2\beta-m+s}=1.$$ 
Consequently, \eqref{Hp3 Prop: Time dependent Lyapunov functions in case of polynomially growing diffusions} holds true for any $c<1$.

\textit{Step 2.} We now show that $A$ satisfies Hypothesis \ref{h.2}.
Fix $0<t\le 1,\,x\in \R^d$, $0<a_0<a<b<b_0<t$ and  $k>d+2$. Let $(t,y)\in [a_0,b_0]\times\R^d$.
We assume that $|y|\geq 1$; otherwise, in a neighborhood of the origin, all the quantities we are going to estimate are obviously bounded.
First, since $\varepsilon<\varepsilon_1$, we have that
$$w\le c_1w^{\frac{k-2}{k}}W_1^{\frac{2}{k}}$$
with $c_1=1$. Second, an easy computation shows that
\begin{align}\label{eq1: Kernel estimates in case of polynomially growing diffusions}
\frac{|Q(y)\nabla w(t,y)|}{w(t,y)^{\frac{k-1}{k}}W_1(t,y)^{\frac{1}{k}}}
&= \varepsilon \beta t^\alpha |y|^{\beta-1}(1+|y|^m) e^{-\frac{1}{k}(\varepsilon_1-\varepsilon)t^\alpha |y|^\beta}\notag\\
&\leq 2\varepsilon \beta t^\alpha |y|^{\beta+m-1} e^{-\frac{1}{k}(\varepsilon_1-\varepsilon)t^\alpha |y|^\beta}.
\end{align}
We make use of the following remark: since the function $t\mapsto t^p e^{-t}$ on $(0,\infty)$ attains its maximum at the point $t=p$, then for $\tau, \gamma, z>0$ we have
\begin{equation}\label{eq2: Kernel estimates in case of polynomially growing diffusions}
z^\gamma e^{-\tau z^\beta}
=\tau^{-\frac{\gamma}{\beta}}(\tau z^\beta)^\frac{\gamma}{\beta} e^{-\tau z^\beta}
\leq \tau^{-\frac{\gamma}{\beta}} \left( \frac{\gamma}{\beta}\right)^\frac{\gamma}{\beta} e^{-\frac{\gamma}{\beta}}
=: C(\gamma,\beta) \tau^{-\frac{\gamma}{\beta}}.
\end{equation}
Applying \eqref{eq2: Kernel estimates in case of polynomially growing diffusions} to the inequality \eqref{eq1: Kernel estimates in case of polynomially growing diffusions} with $z=|y|$, $\tau=\frac{1}{k}(\varepsilon_1-\varepsilon)t^\alpha$, $\beta=\beta$ and $\gamma=\beta+m-1>0$ yields
\begin{align*}
\frac{|Q(y)\nabla w(t,y)|}{w(t,y)^{\frac{k-1}{k}}W_1(t,y)^{\frac{1}{k}}}
\leq 2C(\beta+m-1, \beta)\varepsilon \beta t^\alpha \left[\frac{1}{k}(\varepsilon_1-\varepsilon)t^\alpha\right]^{-\frac{\beta+m-1}{\beta}} 
\leq \overline{c} t^{-\frac{\alpha (m-1)}{\beta}}
\leq \overline{c} t^{-\frac{\alpha m}{\beta}}.
\end{align*}
Thus, we choose $c_2=\overline{c} t^{-\frac{\alpha m}{\beta}}$, where $\overline{c}$ is a universal constant.
Similarly,
\begin{align*}
&\frac{|{\div}(Q(y)\nabla w(t,y))|}{w(t,y)^{\frac{k-2}{k}}W_1(t,y)^{\frac{2}{k}}}
\leq \frac{m |y|^{m-1} |\nabla w(t,y)| + (1+|y|^m)|\Delta w|}{w(t,y)^{\frac{k-2}{k}}W_1(t,y)^{\frac{2}{k}}}\\
&\leq \varepsilon \beta t^\alpha \left[ m  |y|^{\beta+m-2}  +2 ((\beta-2)^++d) |y|^{\beta+m-2}+2\varepsilon\beta t^{\alpha} |y|^{2\beta +m-2} \right] e^{-\frac{2}{k}(\varepsilon_1-\varepsilon)t^\alpha |y|^\beta}. 
\end{align*}
As a result, applying \eqref{eq2: Kernel estimates in case of polynomially growing diffusions} to each term, we find that
\begin{align*}
\frac{|{\div}(Q(y)\nabla w(t,y))|}{w(t,y)^{\frac{k-2}{k}}W_1(t,y)^{\frac{2}{k}}}
\leq  &C(\beta, m) \varepsilon \beta t^\alpha \left \lbrace  [m+2((\beta-2)^++d)]  \left[\frac{2}{k}(\varepsilon_1-\varepsilon)t^\alpha\right]^{-\frac{\beta+m-2}{\beta}}\right. \\
&\left. +2\varepsilon\beta t^{\alpha}\left[\frac{2}{k}(\varepsilon_1-\varepsilon)t^\alpha\right]^{-\frac{2\beta+m-2}{\beta}}\right\rbrace
 \leq  \overline{c} t^{-\frac{\alpha(m-2)}{\beta}}\leq \overline{c} t^{-\frac{\alpha m}{\beta}}.
\end{align*}
Therefore, we pick $c_3=\overline{c} t^{-\frac{\alpha m}{\beta}}$. In the same way, we have
\begin{align*}
\frac{|\partial_t w(t,y)|}{w(t,y)^{\frac{k-2}{k}}W_1(t,y)^{\frac{2}{k}}}
= \varepsilon\alpha t^{\alpha-1} |y|^\beta e^{-\frac{2}{k}(\varepsilon_1-\varepsilon)t^\alpha |y|^\beta}
\leq C(\beta) \varepsilon\alpha t^{\alpha-1}\left[\frac{2}{k}(\varepsilon_1-\varepsilon)t^\alpha\right]^{-1}
\leq \overline{c} t^{-1}.
\end{align*}
Then, we take $c_4= \overline{c} t^{-1}$.
Finally,
$$\frac{V(y)^{\frac{1}{2}}}{w(t,y)^{-\frac{1}{k}}W_2(t,y)^{\frac{1}{k}}}
=|y|^\frac{s}{2} e^{-\frac{1}{k}(\varepsilon_2-\varepsilon)t^\alpha |y|^\beta}
\leq C(s,\beta)\left[\frac{1}{k}(\varepsilon_2-\varepsilon)t^\alpha\right]^\frac{-s}{2\beta}
\leq \overline{c}t^{-\frac{\alpha s}{2\beta}},$$
so we set $c_5=\overline{c}t^{-\frac{\alpha s}{2\beta}}.$

\textit{Step 3.} We check the remaining hypotheses of Theorem \ref{theo-estimate} assuming as above that $|y|\geq 1$.
First, we have
\begin{align*}
\frac{|\Delta w(t,y)|}{w(t,y)^{\frac{k-2}{k}}W_1(t,y)^{\frac{2}{k}}}
= \varepsilon\beta t^\alpha \left[(\beta -2+d)  |y|^{\beta-2}  + \varepsilon\beta t^{\alpha} |y|^{2\beta-2}\right] e^{-\frac{2}{k}(\varepsilon_1-\varepsilon)t^\alpha |y|^\beta}.
\end{align*}
Recalling that $\abs{y}\geq 1$ and applying 
\eqref{eq2: Kernel estimates in case of polynomially growing diffusions}, yields 
\begin{align*}
\frac{|\Delta w(t,y)|}{w(t,y)^{\frac{k-2}{k}}W_1(t,y)^{\frac{2}{k}}}
&\leq  \varepsilon\beta t^\alpha \left[((\beta-2)^++d)  |y|^{\beta}  + \varepsilon\beta t^{\alpha} |y|^{2\beta}\right] e^{-\frac{2}{k}(\varepsilon_1-\varepsilon)t^\alpha |y|^\beta}\\
&\leq C(\beta)\varepsilon\beta t^\alpha \left\lbrace ((\beta-2)^++d)\left[\frac{2}{k}(\varepsilon_1-\varepsilon)t^\alpha\right]^{-1}+\varepsilon\beta t^{\alpha}\left[\frac{2}{k}(\varepsilon_1-\varepsilon)t^\alpha\right]^{-2} \right\rbrace
\leq  \overline{c}.
\end{align*}
Thus, \eqref{Further assumption1} is verified by taking $c_6=\overline{c}$. 
To choose the constant $c_7$ in \eqref{Further assumption2}, we let $t_0\in (0,t)$. Then, we get
\begin{align*}
\frac{|Q(y)\nabla W_1(t_0,y)|}{w(t,y)^{-1/k}W_1(t_0, y)W_2(t,y)^{1/k}}
&=\frac{\varepsilon_1\beta t_0^\alpha |y|^{\beta-1}(1+|y|^m)W_1(t_0,y)}{w(t,y)^{-1/k}W_1(t_0, y)W_2(t,y)^{1/k}}\\
&\leq 2\varepsilon_1\beta t^\alpha |y|^{\beta+m-1}e^{-\frac{1}{k}(\varepsilon_2-\varepsilon)t^\alpha |y|^\beta}\\
&\leq 2C(\beta,m)\varepsilon_1\beta t^\alpha\left[\frac{1}{k}(\varepsilon_2-\varepsilon)t^\alpha\right]^{-\frac{\beta+m-1}{\beta}}\\
&\leq \overline{c} t^{-\frac{\alpha(m-1)}{\beta}}
\leq \overline{c} t^{-\frac{\alpha m}{\beta}}.
\end{align*} 
Consequently, we set $c_7=\overline{c} t^{-\frac{\alpha m}{\beta}}$. Finally, we observe that \eqref{Further assumption3} is clearly satisfied.

To sum up, the constants $c_1, \dots, c_7$ are the following:
\begin{align*}
&c_1=1, 
&&c_2=c_3=c_7=\overline{c} t^{-\frac{\alpha m}{\beta}},
&&c_4= \overline{c} t^{-1},\\
&c_5=\overline{c}t^{-\frac{\alpha s}{2\beta}},
&&c_6=\overline{c}.
\end{align*}

\textit{Step 4.} We are now ready to apply Theorem \ref{theo-estimate}. Thus, there is a positive constant $C>0$ 
depending only on $d$ and $k$ such that
\begin{eqnarray}\label{eq3: Kernel estimates in case of polynomially growing diffusions}
w(t,y)p(t,x,y) &\le & C\left[c_1^{\frac{k}{2}}\sup_{s\in (a_0,b_0)}\xi_{W_1}(s,x)+\left(
c_2^k+\frac{c_1^{\frac{k}{2}}}{(b_0-b)^{\frac{k}{2}}}+c_3^{\frac{k}{2}}
+c_4^{\frac{k}{2}}+c_6^{\frac{k}{2}}\right)\int_{a_0}^{b_0}\xi_{W_1}(s,x)\,ds\right.\notag\\
& & \quad \quad \left. +(c_5^k+c_2^\frac{k}{2}c_7^\frac{k}{2})\int_{a_0}^{b_0}\xi_{W_2}(s,x)\,ds\right]
\end{eqnarray}
for all $(t,y)\in (a,b)\times \R^d$ and fixed $x\in \R^d$.
We set $a_0=\frac{t}{4},b=\frac{t}{2}$ and $b_0=\frac{3t}{4}$.
Moreover, by Proposition \ref{Prop: Time dependent Lyapunov functions in case of polynomially growing diffusions}, there are two constants $H_1$ and $H_2$ not depending on $a_0$ and $b_0$ such that $\xi_{W_j}(s,x)\leq H_j$ for all $(s,x)\in [0,1]\times\R^d$, so 
\begin{equation*}
\int_{a_0}^{b_0}\xi_{W_j}(s,x)\,ds
\leq H_j(b_0-a_0)= \frac{t}{2} H_j.
\end{equation*}
If we now replace in \eqref{eq3: Kernel estimates in case of polynomially growing diffusions} the values of the constants $c_1, \dots, c_7$ determined in Step 3, we use the previous inequality and we consider $C$ as a positive constant that can vary from line to line, we obtain
\begin{equation}\label{eq4: Kernel estimates in case of polynomially growing diffusions}
w(t,y)p(t,x,y)\leq C \left[ t^{1-\frac{\alpha m}{\beta}k}+ t^{1-\frac{k}{2}}+t^{1-\frac{\alpha s}{2\beta}k}\right].
\end{equation}
We note that, since $\alpha>\frac{\beta}{\beta+m-2}$, $s>|m-2|$ and $\beta=\frac{s-m+2}{2}$, it follows that
$$\frac{\alpha(m\vee \frac{s}{2})}{\beta}> \frac{m\vee \frac{s}{2}}{\beta+m-2}>\frac{s}{2(\beta+m-2)}= \frac{s}{s+m-2}>\frac{1}{2}.$$
Hence,
$$t^{1-\frac{k}{2}}<t^{1-\frac{\alpha(m\vee \frac{s}{2})}{\beta}k}.$$
Consequently, by \eqref{eq4: Kernel estimates in case of polynomially growing diffusions}, we find that
$$w(t,y)p(t,x,y)\leq C t^{1-\frac{\alpha(m\vee \frac{s}{2})k}{\beta}}= C t^{1-\frac{\alpha(2m\vee s)k}{s-m+2}}.$$
Writing the expression of the weight function $w$ we gain the following inequality:
\begin{equation}\label{eq5: Kernel estimates in case of polynomially growing diffusions}
p(t,x,y)\le C t^{1-\frac{\alpha(2m\vee s)}{s-m+2}k}e^{-\varepsilon t^\alpha |y|_*^{\frac{s-m+2}{2}}}
\end{equation}
for $k>d+2$ and for any $t\in (0,1)$, $x,y\in \R^d$.

\textit{Step 5.} Since $A^*=A$, we have that $p^*(t,x,y)=p(t,y,x)$. Therefore, applying \eqref{eq5: Kernel estimates in case of polynomially growing diffusions} to $p^*(t,y,x)$, we derive that
\begin{equation*}
p(t,x,y)= p^*(t,y,x)\le C t^{1-\frac{\alpha(2m\vee s)}{s-m+2}k}e^{-\varepsilon t^\alpha |x|_*^{\frac{s-m+2}{2}}} 
\end{equation*}
for all $t\in (0,1)$ and $x,y\in \R^d$. Combining this with \eqref{eq5: Kernel estimates in case of polynomially growing diffusions} yields
$$p(t,x,y)=p(t,x,y)^{1/2}p(t,x,y)^{1/2} \leq C t^{1-\frac{\alpha(2m\vee s)}{s-m+2}k}e^{-\frac{\varepsilon}{2} t^\alpha |x|_*^{\frac{s-m+2}{2}}} e^{-\frac{\varepsilon}{2} t^\alpha |y|_*^{\frac{s-m+2}{2}}}$$
for $k>d+2$ and for any $t\in (0,1)$, $x,y\in \R^d$.
\end{proof}

\subsection{Exponentially growing coefficients}\label{S52}

\smallskip\noindent
In this subsection we apply Theorem \ref{theo-estimate} to the case of operators with exponentially diffusion and potential terms.

Let $Q(x)=e^{\abs{x}^m}I$ and $V(x)=e^{\abs{x}^s}$ with $2\leq m<s$. Set
$$w(t,x)=\exp\left( \varepsilon t^\alpha \int_0^{\abs{x}_*} e^{\frac{\tau^\beta}{2}}\, d\tau\right)\hbox{\ and } W_j(t,x)=\exp\left( \varepsilon_j t^\alpha \int_0^{\abs{x}_*} e^{\frac{\tau^\beta}{2}}\, d\tau\right),$$
where $j=1,2$, $\frac{m}{2}+1\leq\beta\leq m$, $0<\varepsilon<\varepsilon_1<\varepsilon_2$ and $\alpha>\frac{2\beta+m-2}{2m}$.

\begin{theo}\label{thm:5-2}
Let $p$ be the integral kernel associated with the operator $A$ with $Q(x)=e^{\abs{x}^m}I$ and $V(x)=e^{\abs{x}^s}$, where $2\leq m<s$. Then
$$p(t,x,y)\leq C t^{1-\frac{k}{2}} \exp({Ckt^{-\alpha}})\exp\left(- \frac{\varepsilon}{2} t^\alpha \int_0^{\abs{x}_*} e^{\frac{\tau^\beta}{2}}\, d\tau\right)\exp\left(- \frac{\varepsilon}{2} t^\alpha \int_0^{\abs{y}_*} e^{\frac{\tau^\beta}{2}}\, d\tau\right),$$
for $k>d+2$ and any $t\in (0,1),\,x,y\in \R^d$.
\end{theo}

\begin{proof}
\textit{Step 1.} We check conditions \eqref{Hp1 Prop: Time dependent Lyapunov functions in case of exponentially growing diffusions}, \eqref{Hp2 Prop: Time dependent Lyapunov functions in case of exponentially growing diffusions} and \eqref{Hp3 Prop: Time dependent Lyapunov functions in case of exponentially growing diffusions} to apply Proposition \ref{Prop: Time dependent Lyapunov functions in case of exponentially growing diffusions} and show that $W_1$ and $W_2$ are time dependent Lyapunov functions for $L=\partial_t+A$. It's clear that \eqref{Hp1 Prop: Time dependent Lyapunov functions in case of exponentially growing diffusions} holds true with $c_e=1$. Moreover, since $s>m$, it follows that
$$\lim_{\abs{x}\to\infty} V(x) \abs{x}^{1-\beta-m} e^{-\abs{x}^\beta-\abs{x}^m}=\lim_{\abs{x}\to\infty} \abs{x}^{1-\beta-m} e^{\abs{x}^s-\abs{x}^\beta-\abs{x}^m}=+\infty$$
and
\begin{align*}
&\limsup_{\abs{x}\to\infty}|x|^{1-\beta-m}e^{-\frac{\abs{x}^\beta}{2}-\abs{x}^m}\left(G\cdot \frac{x}{|x|}-\frac{V}{\varepsilon e^{\frac{\abs{x}^\beta}{2}}}\right)\\
&= \limsup_{\abs{x}\to\infty} \left( m\abs{x}^{-\beta} e^{-\frac{\abs{x}^\beta}{2}} -\frac{1}{\varepsilon}|x|^{1-\beta-m} e^{\abs{x}^s-\abs{x}^\beta-\abs{x}^m}\right)=-\infty.
\end{align*}
Consequently, there exist constants $c, \Lambda>0$ such that \eqref{Hp2 Prop: Time dependent Lyapunov functions in case of exponentially growing diffusions} and \eqref{Hp3 Prop: Time dependent Lyapunov functions in case of exponentially growing diffusions} hold true.
By Proposition \ref{Prop: Time dependent Lyapunov functions in case of exponentially growing diffusions} we conclude that $W_1$ and $W_2$ are time dependent Lyapunov functions. In addition, we also note that Hypothesis \ref{h.1} is verified with
$$Z(x)=\exp\left( \varepsilon_2 \int_0^{\abs{x}_*} e^{\frac{\tau^\beta}{2}}\, d\tau\right).$$

\textit{Step 2.} We prove that $A$ satisfies all the assumptions of Theorem \ref{theo-estimate}. 
Fix $0<t\le 1,\,x\in \R^d$, $0<a_0<a<b<b_0<t$ and  $k>d+2$. Let $(t,y)\in [a_0,b_0]\times\R^d$.
If $|y|\leq 1$, by continuity all the functions we are estimating are bounded by a constant.
Thus, let $|y|\geq 1$.
Since $\varepsilon<\varepsilon_1$, we have that $w\leq W_1$. Hence, inequality
$$w\le c_1w^{\frac{k-2}{k}}W_1^{\frac{2}{k}}$$
holds true with $c_1=1$. After, observing that
\begin{equation}\label{eq1: Kernel estimates in case of exponentially growing diffusions}
\int_0^{\abs{y}} e^{\frac{\tau^\beta}{2}}\, d\tau\geq \int_{\abs{y}-1}^{\abs{y}} e^{\frac{\tau^\beta}{2}}\, d\tau\geq e^{\frac{(\abs{y}-1)^\beta}{2}}
\end{equation}
leads to
\begin{align}\label{eq2: Kernel estimates in case of exponentially growing diffusions}
\frac{|Q(y)\nabla w(t,y)|}{w(t,y)^{\frac{k-1}{k}}W_1(t,y)^{\frac{1}{k}}}
&= \varepsilon t^\alpha \exp\left( \frac{\abs{y}^\beta}{2}+\abs{y}^m-\frac{(\varepsilon_1-\varepsilon)}{k}t^\alpha \int_0^{\abs{y}} e^{\frac{\tau^\beta}{2}}\, d\tau\right)\notag\\
& \leq \varepsilon t^\alpha \exp\left( \frac{\abs{y}^\beta}{2}+\abs{y}^m-\frac{(\varepsilon_1-\varepsilon)}{k}t^\alpha e^{\frac{(\abs{y}-1)^\beta}{2}}\right).
\end{align}
We now consider the function
$$f(r):=\frac{r^\beta}{2}+r^m-\tilde{\varepsilon} t^\alpha e^\frac{(r-1)^\beta}{2},$$
where $r\geq 1$ and $\tilde{\varepsilon}:=(\varepsilon_1-\varepsilon)/{k}$.
Considering that there exists a universal constant $\overline{c}>0$ (that can vary from line to line) depending on $\beta$ and $m$ such that
$$\frac{r^\beta}{2}+r^m\leq \overline{c} e^\frac{(r-1)^\beta}{4}, \quad \forall r\geq 1,$$
we get
$$f(r)\leq \overline{c} e^\frac{(r-1)^\beta}{4}-\tilde{\varepsilon} t^\alpha e^\frac{(r-1)^\beta}{2}.$$
If we set $z=e^\frac{(r-1)^\beta}{2}$ and we compute the maximum of the function $h(z)= \overline{c} \sqrt{z}-\tilde{\varepsilon} t^\alpha z$, we obtain that
$$f(r)\leq \frac{\overline{c}^2}{4\tilde{\varepsilon}}t^{-\alpha}.$$
As a result, by \eqref{eq2: Kernel estimates in case of exponentially growing diffusions} we derive
$$\frac{|Q(y)\nabla w(t,y)|}{w(t,y)^{\frac{k-1}{k}}W_1(t,y)^{\frac{1}{k}}}\leq \varepsilon t^\alpha \exp\left(\frac{\overline{c}^2 }{4\tilde{\varepsilon}}t^{-\alpha}\right).$$
Then, we set $c_2:= \overline{c} t^\alpha \exp(\overline{c}t^{-\alpha})$.
In a similar way, we have that
\begin{align*}
\frac{|{\div}(Q(y)\nabla w(t,y))|}{w(t,y)^{\frac{k-2}{k}}W_1(t,y)^{\frac{2}{k}}}
\leq \Bigg[&(d-1)\varepsilon t^\alpha \frac{1}{\abs{y}} e^{\frac{\abs{y}^\beta}{2}+\abs{y}^m} +m\varepsilon t^\alpha \abs{y}^{m-1} e^{\frac{\abs{y}^\beta}{2}+\abs{y}^m} + \frac{\beta}{2}\varepsilon t^\alpha \abs{y}^{\beta-1} e^{\frac{\abs{y}^\beta}{2}+\abs{y}^m}\\
& +\varepsilon^2 t^{2\alpha}e^{\abs{y}^\beta+\abs{y}^m}\Big] \exp\left(-\frac{2(\varepsilon_1-\varepsilon)}{k}t^\alpha \int_0^{\abs{y}} e^{\frac{\tau^\beta}{2}}\, d\tau\right).
\end{align*}
Using again \eqref{eq1: Kernel estimates in case of exponentially growing diffusions}, we deduce
\begin{align*}
\frac{|{\div}(Q(y)\nabla w(t,y))|}{w(t,y)^{\frac{k-2}{k}}W_1(t,y)^{\frac{2}{k}}}
\leq &(d-1)\varepsilon t^\alpha \exp\left( \frac{\abs{y}^\beta}{2}+\abs{y}^m -\frac{2(\varepsilon_1-\varepsilon)}{k}t^\alpha e^{\frac{(\abs{y}-1)^\beta}{2}} \right)\\
&+m\varepsilon t^\alpha \exp\left( \log\abs{y}^{m-1}+ \frac{\abs{y}^\beta}{2}+\abs{y}^m -\frac{2(\varepsilon_1-\varepsilon)}{k}t^\alpha e^{\frac{(\abs{y}-1)^\beta}{2}}\right)\\
&+ \frac{\beta}{2}\varepsilon t^\alpha \exp\left( \log\abs{y}^{\beta-1}+ \frac{\abs{y}^\beta}{2}+\abs{y}^m -\frac{2(\varepsilon_1-\varepsilon)}{k}t^\alpha e^{\frac{(\abs{y}-1)^\beta}{2}}\right)\\
& +\varepsilon^2 t^{2\alpha} \exp\left( \abs{y}^\beta+\abs{y}^m-\frac{2(\varepsilon_1-\varepsilon)}{k}t^\alpha e^{\frac{(\abs{y}-1)^\beta}{2}}\right).
\end{align*}
Proceeding as above yields
\begin{align*}
\frac{|{\div}(Q(y)\nabla w(t,y))|}{w(t,y)^{\frac{k-2}{k}}W_1(t,y)^{\frac{2}{k}}}
&\leq \left( (d-1) +m+\frac{\beta}{2}\right)\varepsilon t^\alpha  \exp\left(\frac{\overline{c}^2}{8\tilde{\varepsilon}}t^{-\alpha}\right)+ \varepsilon^2 t^{2\alpha}\exp\left(\frac{\overline{c}^2}{8\tilde{\varepsilon}}t^{-\alpha}\right)\\
&\leq \overline{c} \varepsilon t^\alpha \exp\left(\frac{\overline{c}^2}{8\tilde{\varepsilon}}t^{-\alpha}\right).
\end{align*}
Thus, we choose $c_3=\overline{c} t^\alpha\exp(\overline{c} t^{-\alpha})$.
Concerning $c_4$, we have
\begin{align*}
\frac{|\partial_t w(t,y)|}{w(t,y)^{\frac{k-2}{k}}W_1(t,y)^{\frac{2}{k}}}
&= \varepsilon \alpha t^{\alpha-1}\left(\int_0^{\abs{y}} e^{\frac{\tau^\beta}{2}}\, d\tau\right)\exp\left(-\frac{2(\varepsilon_1-\varepsilon)}{k}t^\alpha \int_0^{\abs{y}} e^{\frac{\tau^\beta}{2}}\, d\tau\right)\\
&\leq \varepsilon \alpha \frac{k}{2(\varepsilon_1-\varepsilon)} t^{-1}.
\end{align*}
We take $c_4=\overline{c} t^{-1}$. Repeating the same procedure for the remaining estimates, we get
$c_5=\overline{c} e^{\overline{c}t^{-\alpha}}$ and $c_6=c_7=c_2.$

\textit{Step 3.} As in Theorem \ref{Thm: Kernel estimates in case of polynomially growing diffusions}, we choose $a_0=\frac{t}{4},b=\frac{t}{2}$ and $b_0=\frac{3t}{4}$ and we notice that, by Proposition \ref{Prop: Time dependent Lyapunov functions in case of exponentially growing diffusions}, there are two constants $H_1$ and $H_2$ not depending on $a_0$ and $b_0$ such that
\begin{equation*}
\int_{a_0}^{b_0}\xi_{W_j}(s,x)\,ds
\leq H_j(b_0-a_0)= \frac{t}{2} H_j.
\end{equation*}
Applying Theorem \ref{theo-estimate}, we infer that there exists a positive constant $C>0$ depending only on $d$ and $k$ such that
\begin{eqnarray*}
w(t,y)p(t,x,y) &\le& C\left[c_1^{\frac{k}{2}}\sup_{s\in (a_0,b_0)}\xi_{W_1}(s,x)+\left(
c_2^k+\frac{c_1^{\frac{k}{2}}}{(b_0-b)^{\frac{k}{2}}}+c_3^{\frac{k}{2}}
+c_4^{\frac{k}{2}}+c_6^{\frac{k}{2}}\right)\int_{a_0}^{b_0}\xi_{W_1}(s,x)\,ds\right.\notag\\
& & \quad \quad \left. +(c_5^k+c_2^\frac{k}{2}c_7^\frac{k}{2})\int_{a_0}^{b_0}\xi_{W_2}(s,x)\,ds\right]
\end{eqnarray*}
for all $(t,y)\in (a,b)\times \R^d$ and fixed $x\in \R^d$.
We rewrite the previous inequality taking into account the values of the constants $c_1, \dots, c_7$ found in Step 2, keeping track only of powers of $t$ and absorbing all other constants into the constant $C$:
\begin{align*}
w(t,y)p(t,x,y)
&\leq C\Big[ t^{1+\alpha t} \exp({\overline{c}kt^{-\alpha}}) +t^{1+\frac{\alpha k}{2}} \exp\left(\frac{\overline{c}k}{2}t^{-\alpha}\right) +t^{1-\frac{k}{2}} +t \exp({\overline{c}kt^{-\alpha}})\Big]\\
&\leq C t^{1-\frac{k}{2}} \exp({Ckt^{-\alpha}}).
\end{align*}
Hence,
\begin{equation}\label{eq3: Kernel estimates in case of exponentially growing diffusions}
p(t,x,y)\leq C t^{1-\frac{k}{2}} \exp({Ckt^{-\alpha}})\exp\left(- \varepsilon t^\alpha \int_0^{\abs{y}_*} e^{\frac{\tau^\beta}{2}}\, d\tau\right)
\end{equation}
for $k>d+2$ and for any $t\in (0,1)$, $x,y\in \R^d$, where $C$ depends only on $d, \eta, \beta$ and $m$.

\textit{Step 4.} We conclude the proof by applying inequality \eqref{eq3: Kernel estimates in case of exponentially growing diffusions} to $p^*(t,y,x)$. This is possible because $A^*=A$, so $p^*(t,y,x)=p(t,x,y)$. Then we obtain
$$p^*(t,y,x)\leq C t^{1-\frac{k}{2}} \exp({Ckt^{-\alpha}})\exp\left(- \varepsilon t^\alpha \int_0^{\abs{x}_*} e^{\frac{\tau^\beta}{2}}\, d\tau\right)$$
for all $t\in (0,1)$ and $x,y\in \R^d$. As a consequence, we get the desired inequality as follows:
\begin{align*}
p(t,x,y)&=p(t,x,y)^\frac{1}{2}p^*(t,y,x)^\frac{1}{2}\\
&\leq C t^{1-\frac{k}{2}} \exp({Ckt^{-\alpha}})\exp\left(- \frac{\varepsilon}{2} t^\alpha \int_0^{\abs{x}_*} e^{\frac{\tau^\beta}{2}}\, d\tau\right)\exp\left(- \frac{\varepsilon}{2} t^\alpha \int_0^{\abs{y}_*} e^{\frac{\tau^\beta}{2}}\, d\tau\right),
\end{align*}
for all $t\in (0,1)$ and $x,y\in \R^d$.

\end{proof}

\subsection{Spectral properties and eigenfunctions estimates}\label{S53}

\smallskip\noindent
In this subsection we study some spectral properties of $A_{\min}$ with either polynomial or exponential coefficients. In particular we prove the following result.
\begin{theo}
If $Q(x) = (1+|x|_*^m)I$ and $V(x)=|x|^s$ with $s>|m-2|$ and $m>0$ or $Q(x)=e^{\abs{x}^m}I$ and $V(x)=e^{\abs{x}^s}$, where $2\leq m<s$, then $T_p(t)$ is compact for all $t>0$ and $p\in (1,\infty)$. Moreover the spectrum of the generator of $T_p(\cdot)$ is independent of $p$ for $p\in (1,\infty)$ and consists of a sequence of negative real eigenvalues which accumulates at $-\infty$.
\end{theo}

\begin{proof}
 By \cite[Theorem 1.6.3]{Davies}, it suffices to prove that  $T_2(t)$ is compact for all $t>0$.
 To this purpose let us assume that $Q(x)=(1+|x|^m_*)I$ and $V(x)=|x|^s$ with $s>m-2$ and $m>2$ or $Q(x)=e^{\abs{x}^m}I$ and $V(x)=e^{\abs{x}^s}$, where $2\leq m<s$. Applying \cite[Corollary 1.6.7]{Davies}, one deduces that the $L^2$-realization $A_0$ of $\calA_0:={\div}(Q\nabla )$ has compact resolvent and thus the semigroup $S(t)$ generated by $A_0$ in $L^2(\R^d)$ is compact for all $t>0$, cf. \cite[Theorem 4.29]{EN00}. Since $V\ge 0$ we have $0\le T_2(t)\le S(t)$ for all $t\ge 0$. Applying the Aliprantis-Burkinshaw theorem \cite[Theorem 5.15]{AB06} we obtain the compactness of $T_2(t)$ for all $t>0$.
 
Let us now show the compactness of $T_2(t)$ in the case where $Q(x)=(1+|x|^m_*)I$ and $V(x)=|x|^s$ with $s>|m-2|$ and $0<m\le 2$. The operator $A_{\min}$ can be considered as the sum of the operator
$\widetilde{A}_2u:=(1+|x|^m_*)\Delta u-|x|^su$ and the operator $Bu:=\nabla (1+|x|^m_*)\cdot \nabla u$. From \cite[Proposition 2.3]{Luca-Abde} we know that $B$ is a small perturbation of $\widetilde{A}_2$. Hence, $R(\lambda ,A_{\min})=R(\lambda , \widetilde{A}_2)(I-BR(\lambda ,\widetilde{A}_2))^{-1}$ for all $\lambda \in \rho(\widetilde{A}_2)$. Moreover, by \cite[Proposition 2.10]{Luca-Abde}, we know that $\widetilde{A}_2$ has compact resolvent and hence $A_{\min}$ has compact resolvent too. Since $T_2(\cdot)$ is an analytic semigroup, we deduce that $T_2(t)$ is compact for all $t>0$.
\end{proof}
  
Let us now estimate the eigenfunctions of $A_{\min}$. To this purpose let us note first that, by the semigroup law and the symmetry of $p(t,\cdot ,\cdot)$ for any $t>0$, we have
$$p(t+s,x,y)=\int_{\R^d}p(t,x,z)p(s,y,z)\,dz,\quad t,s>0,\,x,y\in \R^d.$$ Thus,
$$
p(t,x,x)=\int_{\R^d}p\left(\frac{t}{2},x,y\right)^2\,dy,\quad t>0,\,x\in \R^d.
$$
So, if we denote by $\psi$ an eigenfunction of 
$A_{\min}$ associated to the eigenvalue $\lambda$, then H\"older's inequality implies
\begin{eqnarray*}
e^{\lambda \frac{t}{2}} |\psi(x)|&=& |T_2(t/2)\psi(x)|\\
&\le & \int_{\R^d} p\left(\frac{t}{2},x,y\right)|\psi(y)|\,dy\\
&\le & \left(\int_{\R^d}p\left(\frac{t}{2},x,y\right)^2dy\right)^{\frac{1}{2}}\|\psi \|_2\\
&=&p(t,x,x)^{\frac{1}{2}}\|\psi \|_2
\end{eqnarray*}
for any $t>0$ and any $x\in \R^d$. 
Therefore, if we normalize $\psi$, i.e. $\|\psi\|_2=1$, then 
$$|\psi(x)|\le e^{-\lambda \frac{t}{2}} p(t,x,x)^{\frac{1}{2}},\quad t>0,\,x\in \R^d.$$
So, by Theorem \ref{Thm: Kernel estimates in case of polynomially growing diffusions} and Theorem \ref{thm:5-2} we have
\begin{cor}
Let $\psi$ be any normalized eigenfunction of $A_{\min}$. Then,
\begin{enumerate}
\item in the case of polynomially growing coefficients, i.e., 
$Q(x)=(1+|x|^m_*)I$ and $V(x)=|x|^s$, where $s>|m-2|$ and $m> 0$, we have
\[
|\psi(x)| \le c_1e^{-c_2 |x|_*^{\frac{s-m+2}{2}}},\quad x\in \R^d,
\]
for some constants $c_1,c_2>0$;
\item in the case of exponentially growing coefficients, i.e. $Q(x)=e^{\abs{x}^m}I$ and $V(x)=e^{\abs{x}^s}$, where $2\leq m<s$, we have
\[
|\psi(x)| \le c_1\exp\left(-c_2 \int_0^{\abs{x}_*} e^{\frac{\tau^\beta}{2}}\, d\tau\right),\quad x\in \R^d,
\]
for some constants $c_1,c_2>0$.
\end{enumerate}
\end{cor}

\end{document}